\documentclass[10pt, a4paper]{amsart}
\usepackage[margin=3cm]{geometry}

\usepackage{amsmath}
\usepackage{dsfont}
\usepackage{mathrsfs}
\usepackage{amssymb}
\usepackage{hanging}
\usepackage{latexsym,amsfonts,amsthm}

\usepackage[OT2,OT1]{fontenc}
\newcommand\cyr
{
\renewcommand\rmdefault{wncyr}
\renewcommand\sfdefault{wncyss}
\renewcommand\encodingdefault{OT2}
\normalfont
\selectfont
}
\DeclareTextFontCommand{\textcyr}{\cyr}

\usepackage{color}
\usepackage{tikz}
\usepackage{multirow}

\usetikzlibrary{shapes, decorations, arrows, calc, arrows.meta, fit, positioning}
\tikzset{
    -Latex, auto, node distance = 1 cm and 1 cm, semithick,
    state/.style = {ellipse, draw, minimum width = 0.7 cm},
    point/.style = {circle, draw, inner sep = 0.04cm, fill, node contents = {}},
    bidirected/.style = {Latex-Latex, dashed},
    el/.style = {inner sep = 2pt, align = left, sloped}
}

\usepackage{systeme}
\usepackage{tabularx}
\usepackage{tabulary}
\usepackage{float}
\usepackage{subcaption}
\usepackage{mathtools}
\usepackage{stmaryrd}
\usepackage{mathabx}
\usepackage{url}
\usepackage{relsize}
\usepackage[hidelinks]{hyperref}
\usepackage{diagbox}
\usepackage{bm}
\usepackage{xcolor}

\def\jump#1{[\hspace{-2pt}[#1]\hspace{-2pt}]}

\newcommand{\bn}{{\bm n}}
\newcommand{\bt}{{\bm t}}

\newcommand{\bv}{{\bm v}}
\newcommand{\bp}{{\bm p}}

\newcommand{\bbR}{\mathbb{R}}

\newcommand{\p}{\partial}
\newcommand{\nab}{\nabla}

\newcommand{\bmu}{{\bm \mu}}
\newcommand{\calT}{\mathcal{T}}
\newcommand{\calE}{\mathcal{E}}

\newcommand{\calF}{\mathcal{F}}
\newcommand{\calK}{\mathcal{K}}

\newcommand{\calO}{\mathcal{O}}
\newcommand{\calX}{\mathcal{X}}
\newcommand{\calP}{\mathcal{P}}

\newcommand{\Grad}{\nabla\!\!\!\!\nabla}
\newcommand{\hesstwo}{\nab_\Gamma^2}
\newcommand{\bg}{{\bm g}}

\newcommand{\bnu}{{\bm \nu}}
\newcommand{\bchi}{{\bm \chi}}

\def\jump#1{[\hspace{-2pt}[#1]\hspace{-2pt}]}

\newtheorem{theorem}{Theorem}[section]
\newtheorem{lemma}[theorem]{Lemma}

\newtheorem{corollary}[theorem]{Corollary}

\theoremstyle{definition}

\theoremstyle{remark}
\newtheorem{remark}[theorem]{Remark}
\theoremstyle{example}

\allowdisplaybreaks[3]
\numberwithin{equation}{section}

\newcommand{\tbar}[1]{{\left\vert\kern-0.25ex\left\vert\kern-0.25ex\left\vert #1 
    \right\vert\kern-0.25ex\right\vert\kern-0.25ex\right\vert}}

\begin{document}

\parindent=0in

\title[A TraceFEM $C^0$ IP Method for the Surface Biharmonic]{A TraceFEM $C^0$ Interior Penalty Method for the Surface Biharmonic Equation}

\author[M.~Neilan]{Michael Neilan}
\address{Department of Mathematics, University of Pittsburgh, Pittsburgh, PA 15260 USA}
\email{neilan@pitt.edu}

\author[H.~Wan]{Hongzhi Wan}
\address{Department of Mathematics, University of Pittsburgh, Pittsburgh, PA 15260 USA}
\email{how41@pitt.edu}

\thanks{The authors were supported in part through the NSF, grant DMS-2309425.}
\date{}
\maketitle

\begin{abstract}
We construct and analyze a  TraceFEM discretization for the surface biharmonic problem.
The method utilizes standard quadratic Lagrange finite element spaces defined on a three-dimensional background mesh and a symmetric $C^0$ interior penalty formulation posed on a second-order polyhedral approximation of the surface.
Stability is achieved through a combination of surface edge penalties and bulk-facet penalization of gradient and Hessian jumps. 
We prove optimal first-order convergence in a discrete $H^2$ norm
and quadratic convergence in the $L^2$ norm. 
\end{abstract}

\thispagestyle{empty}

\section{Introduction}

TraceFEM has become a widely used unfitted 
    method for surface PDEs due to its conceptual simplicity
    and robustness with respect to surface geometry. 
    In contrast to surface finite element methods (SFEMs), which 
    build discrete spaces directly on a surface triangulation, TraceFEM
    employs standard volumetric finite element spaces defined on a fixed
    three-dimensional background mesh and poses the discrete formulation
on the induced two-dimensional surface mesh.
    This approach enables
the use of conventional finite element spaces, and it is particularly well-suited
for problems involving evolving surfaces as well as problems with coupling conditions
with the surrounding bulk. This framework is now relatively established and has been 
successfully applied for second-order  elliptic problems, as well
as for the surface Navier-Stokes equations (cf., e.g., \cite{OlshanskiiEtal09,OlshanskiiEtal24,BurmanEtal15,LarsonZahedi20}).

Although several TraceFEM (and SFEM) discretizations have been
developed in the past few decades, substantially fewer
results exist for finite element methods for fourth-order surface PDEs.
Existing contributions include 
 SFEMs based on $C^0$ interior penalty (IP) methods
in \cite{LarssonLarson17,neilan2025c0}, non-conforming SFEMs using
Zienkiewicz-type  elements \cite{WuZhou25}, and SFEMs built
on either Ciarlet-Raviart or Hellan–Herrmann–Johnson mixed formulations
\cite{NitschkeEtal12,BranderEtAl22,Walker22,SunEtal20}.
In contrast, to the best of our knowledge, the only TraceFEM discretization for fourth-order
surface PDEs is given in 
 \cite{reusken2020stream,BranderReusken20} which proposes and 
 analyzes TraceFEM discretizations for the surface stream function
 formulation of the Stokes problem, a fourth-order PDE whose
 principal operator is the surface biharmonic. 
Their method introduces a Ciarlet-Raviart-type mixed finite
element formulation, where the Laplace-Beltrami of the solution is introduced
as an additional unknown. While this construction is relatively straightforward,
using standard Lagrange finite element spaces, it increases
the number of degrees of freedom and leads to a relatively
large saddle-point system.

In this paper, we introduce and analyze a TraceFEM for the surface biharmonic
problem based on a symmetric $C^0$ IP formulation \cite{BrennerSung05,EngelEtal02}.
The discrete space is the standard 
 quadratic Lagrange finite element space defined on a three-dimensional
 background mesh, with consistency enforced through edge-based integral terms.
 In addition to penalization of (surface) gradients over surface edges, analogous to the Euclidean setting, the proposed method also includes penalization 
    of the gradient and Hessian operators across facets in the background mesh. We show
    that these bulk-penalty terms ensure stability of the method, independent
    to how the mesh intersects the  surface.
We prove optimal first-order convergence in a discrete 
$H^2$ norm and quadratic convergence in the $L^2$ norm, with both
estimates taking geometric errors into account.

The proposed method offers several advantages, foremost is
its simplicity:  similar to  \cite{reusken2020stream,BranderReusken20},
it uses standard Lagrange finite element spaces which 
are supported in existing software programs (e.g., NGSolve \cite{ngsolve}).
In addition, the proposed scheme is a primal formulation involving a single unknown.
The resulting algebraic system is symmetric positive definite, in line
with the underlying structure of the elliptic PDE. 
As far as we are 
aware, this is the first primal TraceFEM for fourth-order surface PDEs.

The rest of the paper is organized as follows.
Section \ref{sec-biharmonic} introduces 
the relevant surface differential operators as well as
the surface biharmonic problem. Section 
\ref{sec-prelim} describes the surface approximation, triangulations, scaling estimates
and norm equivalences with respect to  the exact and discrete surfaces. 
In Section \ref{sec-Method}, we present the $C^0$ IP method and establish
coercivity, proving that the discrete problem is well posed.
Section \ref{sec-converge} gives the error analysis,
including optimal first-order convergence 
in a discrete $H^2$ norm and second-order $L^2$ estimates
obtained via a duality argument.
Finally, in Section \ref{sec-numerics} we perform
some numerical experiments which support
the theoretical results.

\section{The surface biharmonic problem}\label{sec-biharmonic}

We assume $\Gamma$ is a simply connected compact oriented smooth surface without boundary in $\bbR^3$. We assume that $\Gamma$ is the level set of a smooth function 
$\phi$, i.e., $\Gamma = \{x\in \bbR^3:\ \phi(x) = 0\}$.
Let $U_{\delta}(\Gamma) = \{x \in \bbR^3: {\rm dist}(x, \Gamma) < \delta\}$ be a neighborhood
of $\Gamma$ with $\delta > 0$ sufficiently small such that the signed distance function $d$ (with $d < 0$ in the interior of $\Gamma$) is well defined. Set $\bn = \nab d$, a column vector,
so that $\bn|_{\Gamma}$ is the outward unit normal of $\Gamma$. We let ${\bf P}$ be the tangential projection, i.e.,
\[
{\bf P} = {\bf I} - \bn \otimes \bn,
\]
where ${\bf I}$ denotes the $3\times 3$ identity matrix.
The closest point projection $\bp:U_\delta\to \Gamma$ is
\[
\bp(x) = x - d(x)\bn(x).
\]
The Weingarten map is denoted by ${\bf H} = \nab^2 d$,
and the principal curvatures of $\Gamma$ are denoted by $\kappa_1$ and $\kappa_2$.

For a function $\psi$ defined on the surface $\Gamma$, its extension
$\psi^e: U_\delta \rightarrow \bbR$ is given by
\[
\psi^e = \psi \circ \bp.
\]
On the other hand, a function $\psi$ defined
on some $\tilde \Gamma \subset U_{\delta}(\Gamma)$,
can be lifted to $\Gamma$ via
\[
\psi^\ell \circ \bp = \psi\quad\text{on }\tilde \Gamma.
\]
We will often drop the subscripts $e$ and $\ell$ for the extensions and lifts
if the context is clear.

The surface gradient of a scalar function
$\psi:\Gamma\to \bbR$ is given by
\[
\nab_\Gamma\psi = {\bf P}\nab\psi^e.
\]
For a vector-valued function $\bv = [v_1, v_2, v_3]^\intercal$, 
the extension is defined componentwise so that $(\bv^e)_j = v_j^e$ and its
Jacobian $\Grad\bv$ satisfies $(\Grad\bv)_{i, j} = \frac{\p v_i}{\p x_j}$ for $i, j = 1, 2, 3$. The surface Jacobian and the surface divergence are as follows:
\[
\Grad_\Gamma\bv = {\bf P}\Grad\bv^e{\bf P}, \qquad {\rm div}_{\Gamma}\bv = {\rm tr}(\Grad_\Gamma \bv).
\]
Combining the above definitions, we define the projected surface Hessian $\hesstwo\psi: \Gamma \rightarrow \bbR^{3\times 3}$
and Laplace-Beltrami operator $\Delta_\Gamma \psi:\Gamma \to \bbR$ as
\[
\nabla_{\Gamma}^{2}\psi = \Grad_{\Gamma}\nabla_{\Gamma}\psi,\quad\Delta_\Gamma\psi = {\rm div}_{\Gamma}\nab_{\Gamma}\psi = {\rm tr}(\nab_\Gamma^2\psi).
\]

Given a source function $f:\Gamma\to \bbR$ 
with vanishing mean value $\int_\Gamma f = 0$, we consider
the problem of finding  $u: \Gamma \rightarrow \bbR^3$ satisfying $\int_\Gamma u = 0$ such that
\begin{equation} \label{eqn:SBiharmonicProblem}
\Delta_\Gamma^2u
= f \quad \text{on } \Gamma,
\end{equation}
where  $\Delta_\Gamma^2u := \Delta_\Gamma(\Delta_\Gamma u)$.


\section{Preliminaries}\label{sec-prelim}

\subsection{Approximate Geometries and Meshes}
Let $\calT^{\rm ext}_h$ be a quasi-uniform mesh of shape-regular simplices of a polygonal domain containing $U_{\delta}(\Gamma)$ with $h \approx {\rm diam}(T)\ \forall T\in \calT^{ext}$. Let $\phi_h$ be the  piecewise linear Lagrange interpolant
of the level set $\phi$, and let $\Gamma_h = \{x\in \bbR^3:\ \phi_h(x)=0\}$
be a polyhedral approximation to $\Gamma$. By standard approximation
properties of the Lagrange interpolant, we have
\begin{equation}\label{eqn:GeomApprox}
\|d\|_{L^\infty(\Gamma_h)} \lesssim h^{2}, \qquad \|\bn - \bn_h\|_{L^\infty(\Gamma_h)} \lesssim h,
\end{equation}
where $\bn_h$ is the outward unit normal of $\Gamma_h$. Here, we use the notation $A \lesssim B$ to mean $A \le cB$ for some constant $c > 0$ independent of $h$ as well as $A \approx B$ to mean $A \lesssim B$ and $B \lesssim A$. The constant $c$ is also
independent on how the surface $\Gamma_h$ intersects 
the background mesh $\calT^{\rm ext}_h$.
Above and throughout this paper, 
we assume $h$ is sufficiently small so that $\Gamma_h\subset U_\delta$, in particular,
the closest point projection is well defined on $\Gamma_h$.
We define the tangential projection with respect
to $\Gamma_h$ as ${\bf P}_h = {\bf I}- \bn_h\otimes \bn_h$
and note that, due to \eqref{eqn:GeomApprox},
\begin{equation}\label{eqn:PPhDiff}
\|{\bf P}-{\bf P}_h\|_{L^\infty(\Gamma_h)}\lesssim h.
\end{equation}

 We define the active mesh as
the set of simplices in $\calT_h^{\rm ext}$ intersecting the approximate surface $\Gamma_h$:
\[
\calT_h = \{T \in \calT^{\rm ext}_h: T \cap \Gamma_h \neq \varnothing\},\qquad
\Omega_h = \bigcup_{T\in \calT_h} T.
\]
Let $\calF_h$ denote the set of two-dimensional faces in $\calT_h$, and define
the induced triangulation of $\Gamma_h$ as
\begin{align*}
\calK_h
& = \{K = T \cap \Gamma_h: T \in \calT_h\}.
\end{align*}
The set of one-dimensional edges in $\calK_h$ is denoted by $\calE_h$, i.e.,
\begin{align*}
\calE_h
& = \{E = \partial K^+ \cap \partial K^-:\ K^+\neq K^-,\ K^+, K^- \in \calK_h\},
\end{align*}
where $\partial K^\pm$ denote the boundaries of $K^{\pm}$. 
If $F = \p T^+\cap \p T^-\in \calF_h$, then we use the notation
$\bnu_F^{\pm}$ for the outward unit normal of $\p T^{\pm}$ restricted to $F$.
Likewise, for $E = \p K^+\cap \p K^-$, $\bmu_E^{\pm}$ is the unit co-normal of $\p K^{\pm}$
restricted to $E$. We shall use the convention $\bnu_F = \bnu_F^+$ and $\bmu_E = (\bmu_E^+-\bmu^-_E)/(1-\bmu^+_E\cdot \bmu^-_E)$ (cf.~\cite{LarssonLarson17}) throughout the paper.
We also let $\bt_E$ denote a unit vector tangent to the edge $E$,
and set $\bmu_h:\calE_h\to \bbR^3$ such that $\bmu_h|_E = \bmu_E$
for all $E\in \calE_h$.

Analogous mesh quantities under
the image of the closest point projection
are denoted by the superscript $\ell$. For example,
\[
\calK_h^\ell = \{\bp(K):\ K\in \calK_h\}
\]
is the induced surface triangulation on the exact surface $\Gamma$,
and $\calE_h^\ell = \{\bp(E):\ E\in \calE_h\}$ is the set of edges
of $\calK_h^\ell$. We use the notation $K^\ell = \bp(K)$ 
and $E^\ell = \bp(E)$ for $K\in \calK_h$ and $E\in \calE_h$,
 let $\bmu_{E^\ell}$ be the unit co-normal of $E^\ell$,
 and set $\bmu_\ell:\calE_h^\ell\to \bbR^3$ such that $\bmu_\ell|_{E^\ell} = \bmu_{E^\ell}$.
Note that $\bmu_{E^\ell} \neq \bmu^\ell_E:=\bp(\bmu_E)$ in general.

On an edge $E = \p K^+\cap \p K^-\in \calE_h$
and face $F = \p T^+\cap \p T^-\in \calF_h$, we define the jump 
of functions $v$ and $w$ defined on $\Gamma_h$ and $\Omega_h$ respectively as
\begin{align*}
[v]|_E = v^+|_E-v^-|_E,\qquad [w]|_F = w^+|_F-w^-|_F,
\end{align*}
where $v^{\pm}$ and $w^{\pm}$ are the restrictions of $v$ and $w$ to $K^{\pm}$ and $T^{\pm}$, respectively.
Note that, by the definition of $\bmu_E$, $[\nab_{\Gamma_h} v] \cdot \bmu_E|_E = \nab_{\Gamma_h} v^+\cdot \bmu_E^+|_E+\nab_{\Gamma_h} v^-\cdot \bmu_E^-|_E$.
The averages of $v$ and $w$ across $E$ and $F$ are respectively given by
\begin{align*}
\{v\}|_E = \frac12 (v^+ +v^-)|_E,\qquad \{w\}_F = \frac12 (w^++ w^-)|_F.    
\end{align*}

Throughout this work, we use the notation
$(\cdot,\cdot)_D$ and $\|\cdot\|_D$ to denote the $L^2$-inner product
and $L^2$-norm with respect to the domain $D\subset \bbR^d\ (d=2,3)$.
For a set of geometric quantities $\calP_h$, we let $(\cdot,\cdot)_{\calP_h}
 = \sum_{P\in \calP_h} (\cdot,\cdot)_P$ 
denote the piecewise $L^2$-inner product with respect to $\calP_h$.

For a scalar-valued function $u: \Gamma \rightarrow \bbR$, 
its $W^{m,p}$ surface Sobolev norm is given by ($m\in \mathbb{N}_0,\ 1\le p<\infty)$
\[
\|u\|_{W^{m, p}(\Gamma)}^p
= \sum_{k = 0}^m\|\nab_\Gamma^ku\|_{L^p(\Gamma)}^p.
\]
For vector-valued function $\bchi: \Gamma \rightarrow \bbR^3$ with $\bchi = (\chi_1,\chi_2,\chi_3)^\intercal$, we set
\[
\|\bchi\|_{W^{m, p}(\Gamma)}^p
= \sum_{j = 1}^3\|\chi_j\|_{W^{m, p}(\Gamma)}^p.
\]
We  denote the piecewise $W^{m,p}$-space over a 
set of geometric entities $\calP_h$ as
\[
W^{m,p}(\calP_h) = \prod_{P\in \calP_h} W^{m,p}(P),
\]
with corresponding norm
$\|\cdot\|_{W^{m,p}(\calP_h)}^p = \sum_{P \in \calP_h}\|\cdot\|_{W^{m,p}(P)}^p$.
Similarly, $\|\cdot\|_{W^{m,p}(\partial\calP_h)}^p = \sum_{P \in \calP_h}\|\cdot\|_{W^{m,p}(\partial P)}^p$. 
This notation is also adopted for semi-norms, and we use the standard conventions $H^m = W^{m,2}$ and $W^{0,p} = L^p$. Furthermore, in the case $m=0$ and $p=2$, we write
$\|\cdot\|_{\calP_h} = \|\cdot\|_{L^2(\calP_h)}$
and $\|\cdot\|_{\p\calP_h} = \|\cdot\|_{L^2(\p \calP_h)}$

%

For  $\tau > 0$ and a point $x$ on $\Gamma$, we set $D_\tau(x) =\{y \in \Gamma: |x - y| < \tau\}$,
and
\[
\calK_{\tau, x} = \{K \in \calK_h: K^\ell \cap D_\tau(x) \neq \varnothing\}, \quad \calT_{\tau, x} = \{T \in \calT_h: T \cap \Gamma_h \in \calK_{\tau, x}\}.
\]

Note that the surface mesh $\calK_h$ can
have arbitrarily small elements and is generally not 
shape regular. Nonetheless, the following result from (cf. \cite[(5.1)--(5.4)]{burman2017cut}), essentially states that, for a triangle
in $\calK_h$ with a large aspect ratio, there exists 
a nearby triangle that is shape-regular.
\begin{lemma}[Large Intersection Covering Property]
There exists a set of points $\calX_h\subset \Gamma$ 
and constant $c>0$ such that
$\{\calK_{ch, x}, x \in \calX_h\}$ and $\{\calT_{c h, x}, x \in \calX_h\}$ are coverings of $\calK_h$ and $\calT_h$, respectively, with the following properties for $h$ sufficiently small :
\begin{itemize}
\item The number of sets containing a given point $y$ is uniformly bounded,
\[
\#\{x \in \calX_h: y \in \calT_{ch, x}\} \lesssim 1 \qquad \forall y \in \bbR^3.
\]

\item The number of elements in each set $\calT_{ch, x}$ is uniformly bounded,
\[
\#\calT_{ch, x} \lesssim 1 \qquad \forall x \in \calX_h,
\]
and each element in $\calT_{ch, x}$ shares at least one face with another element in $\calT_{ch, x}$.

\item For all $x \in \calX_h$ there exists an element $T_x \in \calT_{ch, x}$ that has a large intersection with $\Gamma_h$ in the sense that
\begin{align}
|T_x|
& \sim h|T_x \cap \Gamma_h| = h|K_x| \qquad \forall x \in \calX_h. \label{large-intersect}
\end{align}
\end{itemize}
\end{lemma}


\subsection{Norm Correspondences}
We set $\mu_h = \bn \cdot \bn_h(1 - d\kappa_1)(1 - d\kappa_2)$,
and for each $E\in \calE_h$,
$\mu_E = |({\bf P}-d{\bf H})\bt_E|$.
We then have the change of variable formulas \cite{Demlow09,CockburnDemlow16} 
\begin{align*}
\int_{\Gamma_h} v^e \mu_h &= \int_\Gamma v\quad \forall v\in L^1(\Gamma),\qquad
\int_{E} v^e \mu_E = \int_{E^\ell} v\quad \forall v\in L^1(E^\ell).
\end{align*}
Because $|1-\mu_h| = \calO(h^2)$ and $|1-\mu_E| = \calO(h^2)$, we have
\begin{align}
\|v^e\|_{L^p(K)} \approx \|v\|_{L^p(K^\ell)}, \qquad \|v^e\|_{L^p(\p K)}\approx \|v\|_{L^p(\p K^\ell)} \qquad \forall v\in W^{1,p}(K^\ell). \label{250921-1}
\end{align}
Likewise, since (cf.~\cite{demlow2007adaptive})
\begin{align}
\nab_{\Gamma_h} v^e
& = {\bf P}_h({\bf P} - d {\bf H})(\nab_\Gamma v)^e, \label{250914-1}
\end{align}
we have
\begin{align}
        \label{eqn:GradEquivBound}
\|\nab_{\Gamma_h} v^e\|_{L^p(K)}\approx \|\nab_\Gamma v\|_{L^p(K^\ell)},\quad
\|\nab_{\Gamma_h} v^e\|_{L^p(\p K)}\approx \|\nab_\Gamma v\|_{L^p(\p K^\ell)}\quad
\forall v\in W^{2,p}(K^\ell).
\end{align}
There also holds \cite{LarssonLarson17,neilan2025c0},
\begin{align*}
\nab_{\Gamma_h}^2v^e
& = {\bf P}_h(\nab_\Gamma^2v)^e{\bf P}_h
- d{\bf P}_h{\bf H}(\nab_\Gamma^2v)^e{\bf P}_h - d{\bf P}_h(\nab_\Gamma^2v)^e{\bf H}{\bf P}_h + d^2{\bf P}_h{\bf H}(\nab_\Gamma^2v)^e{\bf H}{\bf P}_h \\
& \qquad - ({\bf P}_h\bn) \otimes ({\bf H}(\nab_\Gamma v)^e){\bf P}_h - ({\bf P}_h{\bf H}(\nab_\Gamma v)^e) \otimes ({\bf P}_h\bn)
 - d{\bf P}_h\nab{\bf H}(\nab_\Gamma v)^e{\bf P}_h,
\end{align*}
where for a vector valued function $\bv=[v_1,v_2,v_3]^\intercal$,
the quantity $\nab {\bf H} \bv$ is a $3 \times 3$ matrix with entries
\begin{align}
(\nab{\bf H}\bv)_{i, j}
& = \sum_{k = 1}^3\frac{\p H_{i, k}}{\p x_j}v_k \qquad i, j = 1, 2, 3. \label{250909-3}
\end{align}
It thus follows that, for all $v \in W^{3,p}(K^\ell)$,
\begin{equation}
    \label{eqn:HessEquivBound}
\begin{split}
\|\nab_{\Gamma_h}^2 v^e\|_{L^p(K)}
&\lesssim \|\nab_{\Gamma}^2 v\|_{L^p(K^\ell)}
+ h \|\nab_\Gamma v\|_{L^p(K^\ell)},\\
\|\nab_{\Gamma_h}^2 v^e\|_{L^p(\p K)}
&\lesssim \|\nab_{\Gamma}^2 v\|_{L^p(\p K^\ell)}
+ h \|\nab_\Gamma v\|_{L^p(\p K^\ell)},
\end{split}
\end{equation}
and for $v\in W^{3,p}(K)$,
\begin{equation}
    \label{eqn:HessEquivBound2}
\begin{split}
\|\nab_{\Gamma}^2 v^\ell\|_{L^p(K^\ell)}
&\lesssim \|\nab_{\Gamma_h}^2 v\|_{L^p(K)}
+ h \|\nab_{\Gamma_h} v\|_{L^p(K)},\\
\|\nab_{\Gamma}^2 v^\ell\|_{L^p(\p K^\ell)}
&\lesssim \|\nab_{\Gamma_h}^2 v\|_{L^p(\p K)}
+ h \|\nab_{\Gamma_h} v\|_{L^p(\p K)}.
\end{split}
\end{equation}

Next, by virtue of $\Delta_{\Gamma_h}v^e = {\rm tr}(\nab_{\Gamma_h}^2v^e)$ 
and $\Delta_{\Gamma}v = {\rm tr}(\nab_{\Gamma}^2v)$ 
we have
\begin{equation}\label{LapChain}
\begin{split}
\Delta_{\Gamma_h}v^e
& = (\Delta_\Gamma v)^e - \bn_h^\intercal(\nab_\Gamma^2v)^e\bn_h \\
& \qquad - 2d({\rm tr}({\bf H}(\nab_\Gamma^2v)^e) - \bn_h^\intercal{\bf H}(\nab_\Gamma^2v)^e\bn_h) \\
& \qquad + d^2({\rm tr}({\bf H}(\nab_\Gamma^2v)^e{\bf H}) - \bn_h^\intercal{\bf H}(\nab_\Gamma^2v)^e{\bf H}\bn_h) \\
& \qquad + 2(\bn \cdot \bn_h)\bn_h^\intercal{\bf H}(\nab_\Gamma v)^e - d{\rm tr}({\bf P}_h\nab{\bf H}(\nab_\Gamma v)^e).
\end{split}
\end{equation}
Therefore
\begin{equation}
    \label{eqn:LapEquivBound}
\begin{split}
\|\Delta_{\Gamma_h} v^e\|_{L^p(K)}&\lesssim \|\Delta_\Gamma v\|_{L^p(K^\ell)}+h^2 \|\nab^2_\Gamma v\|_{L^p(K^\ell)}+ h \|\nab_\Gamma v\|_{L^p(K^\ell)},\\
\|\Delta_{\Gamma_h} v^e\|_{L^p(\p K)}&\lesssim \|\Delta_\Gamma v\|_{L^p(\p K^\ell)}+h^2 \|\nab^2_\Gamma v\|_{L^p(\p K^\ell)}+ h \|\nab_\Gamma v\|_{L^p(\p K^\ell)},
\end{split}
\end{equation}
and
\begin{equation}
    \label{eqn:LapEquivBound2}
\begin{split}
\|\Delta_{\Gamma} v^\ell\|_{L^p(K^\ell)}&\lesssim \|\Delta_{\Gamma_h} v\|_{L^p(K)}+h^2 \|\nab^2_{\Gamma_h} v\|_{L^p(K)}+ h \|\nab_{\Gamma_h} v\|_{L^p(K)},\\
\|\Delta_{\Gamma} v^\ell\|_{L^p(\p K^\ell)}&\lesssim \|\Delta_{\Gamma_h} v\|_{L^p(\p K)}+h^2 \|\nab^2_{\Gamma_h} v\|_{L^p(\p K)}+ h \|\nab_{\Gamma_h} v\|_{L^p(\p K)}.
\end{split}
\end{equation}

We end this subsection with 
an estimate that relates the $L^p$ norm
 on the bulk domain
to the $L^p$ norm on the surface.
\begin{lemma} \label{lem:deltaLem}
For all $v \in W^{m, p}(\Gamma)$ ($m\in \mathbb{N}_0,\ 1 \leq p < \infty$) the following bound holds
\begin{align*}
\|D^\mu u^e\|_{L^p(\Omega_h)}
& \lesssim h^{\frac{1}{p}}\|u\|_{W^{m, p}(\Gamma)}, \qquad |\mu| = m.
\end{align*}
In particular,
\begin{align}
\|D^\mu u^e\|_{\Omega_h}
& \lesssim h^{\frac{1}{2}}\|u\|_{H^m(\Gamma)}, \label{*} \\
\|D^\mu u^e\|_{L^1(\Omega_h)}
& \lesssim h\|u\|_{W^{m, 1}(\Gamma)}. \label{250728-1}
\end{align}
\end{lemma}

\begin{proof}
The proof for the case $p=2$ is given 
in \cite[Lemma 3.1]{reusken2015analysis}.
The arguments given there also trivially
extend to general exponent $p\in [1,\infty)$,
and therefore the proof is omitted.
\end{proof}

\subsection{Finite Element Spaces}
For $D\subset \bbR^3$, let $\mathbb{P}_k(D)$ denote
the space of polynomials of degree $\le k$ with domain $D$.
We use the notation
\[
\mathbb{P}_k^{{\rm dc}}(\calT_h) = \prod_{T\in \calT_h} \mathbb{P}_k(T),\qquad
\mathbb{P}_k^{{\rm c}}(\calT_h) = \mathbb{P}_k^{{\rm dc}}(\calT_h)\cap H^1(\Omega_h),
\]
that is, $\mathbb{P}_k^{{\rm dc}}(\calT_h)$ is the space
of discontinuous $k$th-degree piecewise polynomials on $\calT_h$,
and $\mathbb{P}_k^{{\rm c}}(\calT_h)$ is the Lagrange finite element space,
consisting of globally continuous piecewise polynomials of degree $\le k$.
We set
\[
V_h = \mathbb{P}_2^{{\rm c}}(\calT_h),\qquad V_{h,0} = \left\{v\in V_h:\ \int_{\Gamma_h} v=0\right\}
\]
as the finite element spaces used in the $C^0$ interior penalty method below.

\subsection{Scaling Estimates I}
We will frequently use the following trace inequalities for $v\in W^{2,p}(\calT_h)$ ($1\le p<\infty$)
(cf.~\cite[(12.16)]{ern2021finite}, \cite[(4.4)]{reusken2015analysis}, and \cite[(15)]{hansbo2003finite})
\begin{align}
\|v\|_{L^p(\partial T)}^p
& \lesssim h^{-1}\|v\|_{L^p(T)}^p + h^{p-1}\|\nab v\|_{L^p(T)}^p \qquad \forall T \in \calT_h, \label{trace-inequ-1} \\
\|v\|_{L^p(\Gamma_h \cap T)}^p
& \lesssim h^{-1}\|v\|_{L^p(T)}^p+ h^{p-1}\|\nab v\|_{L^p(T)}^p \qquad \forall T \in \calT_h, \label{trace-inequ-2} \\
\|v\|_{L^p(E \cap F)}^p
& \lesssim h^{-1}\|v\|_{L^p(F)}^p + h^{p-1}\|\nab v\|_{L^p(F)}^p \qquad \forall (E,F) \in \calE_h \times \calF_h. 
 \label{trace-inequ-3}
\end{align}

We also have the following inverse inequalities for $v \in \mathbb{P}_k^{dc}(\calT_h)$
(cf.~\cite[(1.36)--(1.37)]{di2011mathematical} and \cite[Sec. 4]{hansbo2003finite})
\begin{align}
\|\nab v\|_T
& \lesssim h^{- 1}\|v\|_T & & \forall T \in \calT_h, \label{inverse-est-1} \\
\|v\|_{\partial T}
& \lesssim h^{- \frac{1}{2}}\|v\|_T & & \forall T \in \calT_h, \label{inverse-est-2} \\
\|v\|_{K \cap T}
& \lesssim h^{- \frac{1}{2}}\|v\|_T & & \forall (K, T) \in \calK_h \times \calT_h, \label{inverse-est-3} \\
\|v\|_{E \cap F}
& \lesssim h^{- \frac{1}{2}}\|v\|_F & & \forall (E, F) \in \calE_h \times \calF_h. \label{inverse-est-4}
\end{align}

\begin{lemma}\label{lem:NeighborCompare}
Let $v\in \mathbb{P}_k^{\rm dc}(\calT_h)$\ ($k\ge 0$),
and let $T^+,T^-$ be tetrahedra in $\calT_h$ with a common
face $F\in \p T^+\cap \p T^-\in \calF_h$.
Then
\begin{align}
\|v\|_{T^+}^2
& \lesssim \|v\|_{T^-}^2 + \sum_{j = 0}^{k}h^{2j + 1}\left\|\left[\frac{\p^j v}{\p\bnu_F^j}\right]\right\|_F^2, \label{F-1} \\
\|\nab v\|_{T^+}^2
& \lesssim \|\nab v\|_{T^-}^2 + \sum_{j = 0}^{k}h^{2j - 1}\left\|\left[\frac{\p^jv}{\p\bnu_F^j}\right]\right\|_F^2, \label{F-2}
\end{align}
with the hidden constants depending only on the shape-regularity of $\calT_h$ and $k$.
\end{lemma}

\begin{proof}
A proof of \eqref{F-1} can be found in \cite[Lemma 5.1]{massing2014stabilized}, and a proof of \eqref{F-2}
can be found in \cite[(4.5)]{GuzmanMaxim18}.
\end{proof}


\section{$C^0$ Interior Penalty Method}\label{sec-Method}
We define the following bilinear forms and linear form
\begin{align}
\label{eqn:AhDef}A_h(v, w)
& = a_h(v, w) + s_h(v, w), \\
\label{eqn:ahDef}a_h(v, w)
& = (\Delta_{\Gamma_h}v, \Delta_{\Gamma_h}w)_{\calK_h} - (\{\Delta_{\Gamma_h}v\}, \bmu_h \cdot [\nab_{\Gamma_h}w])_{\calE_h} -(\bmu_h \cdot [\nab_{\Gamma_h}v], \{\Delta_{\Gamma_h}w\})_{\calE_h} \\
& \qquad + \frac{\sigma}{h}(\bmu_h\cdot [\nab_{\Gamma_h}v], \bmu_h \cdot [\nab_{\Gamma_h}w])_{\calE_h}, \\
\label{eqn:shDef}s_h(v, w)
& = 
\gamma \left(([\nab v], [\nab w])_{\calF_h} + ([\nab^2v], [\nab^2w])_{\calF_h}\right), \\
l_h(v)
& = (f_h, v)_{\Gamma_h},
\end{align}
where $f_h$ is an approximation to $f$ defined on $\Gamma_h$,
and $\sigma,\gamma>0$ are penalty parameters. 
The $C^0$ interior penalty method is to find $u_h \in V_{h, 0}$ such that
\begin{align}
A_h(u_h, v)
& = l_h(v) \qquad \forall v \in V_{h, 0}. \label{disc-prob}
\end{align}

Associated to the bilinear forms, we define the (semi)norms
\begin{align*} 
\begin{split}
\|v\|_{h}^2
& = \|\Delta_{\Gamma_h}v\|_{\calK_h}^2 + h^{-1}\|\bmu_h \cdot [\nab_{\Gamma_h}v]\|_{\calE_h}^2+\|[\nab v]\|_{\calF_h}^2 + \|[\nab^2v]\|_{\calF_h}^2, \\
\tbar{v}_h^2 & = \|v\|_h^2 +  h \|\Delta_{\Gamma_h} v \|_{\p \calK_h}^2.
 \end{split}
\end{align*}
We also define the analogous bilinear form and linear form defined on the exact surface:
\begin{align*}
a_h^\ell(v, w)
& = (\Delta_{\Gamma}v, \Delta_{\Gamma}w)_{\calK_h^\ell} - (\{\Delta_{\Gamma}v\}, \bmu_{\ell} \cdot [\nab_{\Gamma}w])_{\calE^\ell_h} - (\bmu_{\ell} \cdot [\nab_{\Gamma}v], \{\Delta_{\Gamma}w\})_{\calE_h^\ell} \\
& \qquad + \frac{\sigma}{h}(\bmu_{\ell} \cdot [\nab_{\Gamma}v], \bmu_{\ell} \cdot [\nab_{\Gamma}w])_{\calE^\ell_h}, \\
l(v)
& = (f,v)_{\Gamma},
\end{align*}
along with the corresponding norms:
\begin{align*} 
\begin{split}
\|v\|_{*,h}^2
& = \|\Delta_{\Gamma}v\|_{\calK^\ell_h}^2 + h^{-1}\|\bmu_\ell \cdot [\nab_{\Gamma}v]\|_{\calE^\ell_h}^2+\|[\nab v]\|_{\calF_h}^2 + \|[\nab^2v]\|_{\calF_h}^2, \\
\tbar{v}_{*,h}^2 & = \|v\|_{*,h}^2 +  h \|\Delta_{\Gamma} v \|_{\p \calK^\ell_h}^2
 \end{split}
\end{align*}
\begin{remark}\label{rem:Consistent}
Note that, just like in the Euclidean setting,
the bilinear form $a_h^\ell(\cdot,\cdot)$ is consistent
with the surface biharmonic operator. In particular,
standard integration-by-parts identities 
show that the exact solution to \eqref{eqn:SBiharmonicProblem} 
satisfies  $a_h^\ell(u, v^\ell) = l(v^\ell)$ for all $v \in V_{h, 0}$ provided $u\in H^3(\Gamma)$.
\end{remark}


\subsection{Scaling Estimates II}
The stability and convergence analysis of the $C^0$ IP method requires a few preliminary results, relating norms between the surface and bulk triangulations as well as equivalence norm estimates on $V_h$. First, we have the following Poincare-type inequality for piecewise polynomial spaces.

\begin{lemma}
Let $w \in \mathbb{P}_k^{dc}(\calT_h)$,
and let $\lambda_{\Gamma_h}(w)\in \bbR$ denote
the average of $w$ over $\Gamma_h$.
Then for $h$ sufficiently small, there holds 
\begin{align}
\|w - \lambda_{\Gamma_h}(w)\|_{\calT_h}^2
& \lesssim h\|\nab_{\Gamma_h}w\|_{\calK_h}^2 + h^{2}\|\nab w\|_{\calT_h}^2 + h^{-1}\|[w]\|_{\calF_h}^2.
\label{7.1}
\end{align}
\end{lemma}

\begin{proof}
The proof is a simple generalization of the one given in \cite[Lemma 5.2]{burman2017cut}, and is therefore omitted.
\end{proof}


\begin{lemma} \label{lem:72}
For $v \in V_h$ and $h$ sufficiently small, the following estimate holds:
\begin{align} \label{lem-7.2}
h\|\nab^2v\|_{\calT_h}^2
& \lesssim h^2\|\Delta_{\Gamma_h}v\|_{\calK_h}^2 + \|[\nab^2v]\|_{\calF_h}^2.
\end{align}
\end{lemma}
\begin{proof}
Taking $k = 0$ and $w = \nab^2 v$ in \eqref{7.1} with $v\in V_h$ and setting $A = \lambda_{\Gamma_h}(\nab^2v)$ to be the average of $\nab^2v$, a constant symmetric $3\times 3$ matrix, we have
\begin{align} \label{7.8}
\begin{split}
h\|\nab^2v\|_{\calT_h}^2
& \lesssim h\|A\|_{\calT_h}^2 + h\|\nab^2v - A\|_{\calT_h}^2
 \lesssim h\|A\|_{\calT_h}^2 + \|[\nab^2v]\|_{\calF_h}^2,
\end{split}
\end{align}
where we used that $v$ is a piecewise quadratic polynomial.
%
It remains to estimate the first term $h\|A\|_{\calT_h}^2$. Clearly,
\begin{align}
h\|A\|_{\calT_h}^2
& \lesssim h^2\|A\|_{\calK_h}^2 \lesssim h^2 \|A\|_{\Gamma}^2 \label{7.9}
\end{align}
because $A$ is constant and $|\Gamma| \approx |\Gamma_h|$.\\
{\em Claim:}\ $\|A\|_\Gamma \lesssim \|{\bf P} : A\|_\Gamma$.\\
{\em Proof of claim:}
We show that $A\to \|{\bf P} : A\|_\Gamma$ is a norm on constant symmetric matrices.
Suppose $\|{\bf P} : A \|_\Gamma = 0$, which implies ${\bf P} : A = 0$ on $\Gamma$. Since $\Gamma$ is a closed, orientable surface in $\bbR^3$, its Gauss map is surjective, and so we can consider $\bn = {\bm e}_j\ (j = 1, 2, 3)$, the unit vector whose $j$-th component is one. By considering these three cases we easily conclude ${\rm tr}(A) = 0$. Therefore $0 = {\bf P} : A = {\rm tr}(A) - \bn^\intercal A \bn = - \bn^\intercal A \bn$ for all $\bn \in \mathbb{S}^2$. As $A$ is symmetric,  this implies $A = 0$, and so $A\to \|{\bf P} : A\|_\Gamma$ is a norm. By equivalence of norms, we then have $\|A\|_\Gamma \lesssim \|{\bf P} : A\|_\Gamma$, and the claim is proved. 

Continuing from \eqref{7.9}, we apply the claim and \eqref{eqn:PPhDiff} to obtain
\begin{align*}
h\|A\|_{\calT_h}^2
& \lesssim h^2\|A\|_{\calK_h}^2 \lesssim h^2\|{\bf P} : A\|_\Gamma^2 \lesssim h^2\|{\bf P} : A\|_{\calK_h}^2 \\
& \lesssim h^2\|{\bf P}_{h} : A\|_{\calK_h}^2 + h^2\|({\bf P} - {\bf P}_{h}) : A\|_{\calK_h}^2 \\
& \lesssim h^2\|{\bf P}_{h} : A\|_{\calK_h}^2 + h^4\|A\|_{\calK_h}^2.
\end{align*}
We absorb $h^4\|A\|_{\calK_h}^2$ to the left-hand side for $h$ sufficiently small. Consequently,
\begin{align}
h\|A\|_{\calT_h}^2
\lesssim h^2\|{\bf P}_{h} : A\|_{\calK_h}^2. \label{7.10}
\end{align}
Because $\bn_h$ is constant, there holds
$\Delta_{\Gamma_h} v = {\bf P}_h:\nab^2 v$.
Therefore, using the boundedness of ${\bf P}_{h}$, \eqref{inverse-est-3}, and \eqref{7.1}, it follows that
\begin{align*}
h\|A\|_{\calT_h}^2
& \lesssim h^2\|{\bf P}_{h} : \nab^2v\|_{\calK_h}^2 + h^2\|{\bf P}_{h} : (A - \nab^2v)\|_{\calK_h}^2 \notag \\
& \lesssim h^2\|\Delta_{\Gamma_h} v\|_{\calK_h}^2 + h\|A - \nab^2v\|_{\calT_h}^2 \notag \\
& \lesssim h^2\|\Delta_{\Gamma_h} v\|_{\calK_h}^2 + \|[\nab^2v]\|_{\calF_h}^2,
\end{align*}
which in combination with \eqref{7.8} 
yields the desired inequality \eqref{lem-7.2}.
\end{proof}


\begin{lemma}[Trace Inequality]
The following estimate holds
\begin{align}
h\|\Delta_{\Gamma_h}v\|_{\partial\calK_h}^2
\lesssim \|\Delta_{\Gamma_h}v\|_{\calK_h}^2 + \|[\nab^2v]\|_{\calF_h}^2
\le \|v\|_h^2 \qquad \forall v \in V_h,
\label{7.12}
\end{align}
for $h$ sufficiently small. Consequently, the norms $\|\cdot\|_h$ and $\tbar{\cdot}_h$ are equivalent on $V_h$.
\end{lemma}

\begin{proof}
We start by observing that since both $\bn_h$ and $\nab^2v$ are piecewise constant functions on $\calK_h$, so is $\Delta_{\Gamma_h}v$. Then, successively employing \eqref{inverse-est-4}, \eqref{inverse-est-2}, and applying the large intersection covering property and \eqref{F-1} to $\Delta_{\Gamma_h} v$ yield
\begin{align} \label{7.13}
\begin{split}
h\|\Delta_{\Gamma_h}v\|_{\partial\calK_h}^2
& \lesssim h^{- 1}\|\Delta_{\Gamma_h}v\|_{\calT_h}^2  \\
& \lesssim h^{- 1}\sum_{x \in \calX_h}\|\Delta_{\Gamma_h}v\|_{\calT_{h, x}}^2  \\
& \lesssim h^{- 1}\sum_{x \in \calX_h}\|\Delta_{\Gamma_h}v\|_{T_x}^2 + \|[\Delta_{\Gamma_h}v]\|_{\calF_h}^2  \\
%
&\lesssim \|\Delta_{\Gamma_h} v\|_{\calK_h}^2 + \|[\Delta_{\Gamma_h} v]\|_{\calF_h}^2.
\end{split}
\end{align}
%
Next, we note that
\begin{align*}
\|[\Delta_{\Gamma_h}v]\|_{\calF_h}^2
& = \|{\rm tr}([{\bf P}_{h}\nab^2v])\|_{\calF_h}^2
\leq \|[{\bf P}_{h}\nab^2v]\|_{\calF_h}^2.
\end{align*}
Therefore, by a standard identity for the jump operator, the continuity of ${\bf P}$, the boundedness of ${\bf P}_{h}$, \eqref{eqn:PPhDiff}, and \eqref{inverse-est-2} give us
\begin{align}
\begin{split} \label{251110-1}
\|[{\bf P}_{h}\nab^2v]\|_{\calF_h}^2
& \lesssim \|[{\bf P}_{h}]\{\nab^2v\}\|_{\calF_h}^2 + \|\{{\bf P}_{h}\}[\nab^2v]\|_{\calF_h}^2 \\
& \lesssim \|[{\bf P}-{\bf P}_{h}]\{\nab^2v\}\|_{\calF_h}^2 + \|[\nab^2v]\|_{\calF_h}^2 \\
& \lesssim h\|\nab^2v\|_{\calT_h}^2 + \|[\nab^2v]\|_{\calF_h}^2.
\end{split}
\end{align}
We then use 
\eqref{lem-7.2} to obtain
\begin{align}
\|[\Delta_{\Gamma_h}v]\|_{\calF_h}^2
& \lesssim h^2\|\Delta_{\Gamma_h}v\|_{\calK_h}^2 + \|[\nab^2 v]\|_{\calF_h}^2. \label{eqn:DeltaJumpControl}
\end{align}
Combining \eqref{7.13} and \eqref{eqn:DeltaJumpControl} yields the desired result \eqref{7.12}.
\end{proof}

\subsection{Coercivity and Continuity}

\begin{lemma}[Coercivity]
The discrete bilinear form is coercive with respect to the discrete energy norm, i.e.,
\begin{align}
\|v\|_h^2
& \lesssim A_h(v, v) \qquad \forall v \in V_h, \label{3.11}
\end{align}
provided $\sigma$ is sufficient large and $\gamma>0$.
\end{lemma}

\begin{proof}
Starting from the definition of $A_h$ and applying Cauchy's inequality with $\varepsilon$, we have
\begin{align*}
A_h(v, v)
& \gtrsim \|\Delta_{\Gamma_h}v\|_{\calK_h}^2 - \varepsilon h\|\{\Delta_{\Gamma_h}v\}\|_{\calE_h}^2 + \frac{\sigma - \varepsilon^{-1}}{h}\|\bmu_h \cdot [\nab_{\Gamma_h}v]\|_{\calE_h}^2 +
\gamma(\|[\nab v]\|_{\calF_h}^2 + \|[\nab^2v]\|_{\calF_h}^2), 
\end{align*}
for all $\varepsilon>0$.
Employing \eqref{7.12}, we immediately see that
\[
A_h(v, v) \gtrsim \|v\|_h^2.
\]
when choosing $\varepsilon$ small and $\sigma$
large enough.
\end{proof}

\begin{lemma}[Boundedness]
The discrete bilinear form is bounded in the following sense:
\begin{align}
A_h(v, w)
& \lesssim \tbar{v}_{h}\tbar{w}_h \qquad \forall v, w \in (H^3(\Gamma))^e + V_h. \label{bdd1}
\end{align}
Moreover,
\begin{align} \label{bdd2}
a_h^\ell(v,w)\lesssim \tbar{v}_{*,h}\tbar{w}_{*,h}\qquad \forall v,w\in H^3(\Gamma)+V_h^\ell.
\end{align}

\end{lemma}
\begin{proof}
The continuity estimates follow from the Cauchy-Schwarz inequality and the definitions
of $\tbar{\cdot}_h$ and $\tbar{\cdot}_{*, h}$.
\end{proof}


\subsection{Geometric Consistency Estimates}

Estimating the geometric errors 
in the Trace $C^0$ IP method requires
a few more technical lemmas and notation.
We introduce the following function spaces and norms
\begin{align*}
W
& = \{w \in C^0(\Gamma_h): w|_{K} \in H^3(K) \quad \forall K \in \calK_h\}, \\
W^\ell
& = \{w \in C^0(\Gamma): w|_{K^\ell} \in H^3(K^\ell) \quad \forall K^\ell \in \calK_h^\ell\}, \\
\tbar{w}_{2h}^2
& = \tbar{w}_h^2 +
\|\nab_{\Gamma_h}w\|_{\calK_h}^2 + h^2\|\nab_{\Gamma_h}^2w\|_{\calK_h}^2
+ h\|\nab_{\Gamma_h}w\|_{\partial\calK_h}^2 + h^3\|\nab_{\Gamma_h}^2w\|_{\partial\calK_h}^2,\\ 
\tbar{w}_{*, 2h}^2
& = \tbar{w}_{*,h}^2 
+ \|\nab_{\Gamma}w\|_{\calK^\ell_h}^2 + h^2\|\nab_{\Gamma}^2w\|_{\calK^\ell_h}^2
  + h\|\nab_{\Gamma}w\|_{\partial\calK_h^\ell}^2 + h^3\|\nab_{\Gamma}^2 w\|_{\p \calK_h^\ell}^2.
\end{align*}

We then have the following result
that shows that, for finite element functions, various norms 
defined on the discrete surface are controlled 
by the energy norm $\|{\cdot}\|_h$.
In particular, the norms $\|\cdot\|_h$, $\tbar{\cdot}_h$, and $\tbar{\cdot}_{2h}$
are equivalent on $V_{h,0}$.
Its proof is found in Appendix \ref{sec:ProofOfNormEstimates}.

\begin{lemma}[Norm Equivalences]\label{lem:NormEstimates}
There holds for all $v\in V_{h,0}$,
\begin{equation} \label{eqn:AllAtOnce}
\begin{split}
&h\|\nab v\|_{\calT_h}^2+ \|v\|_{\calK_h}^2 
 + \|\nab_{\Gamma_h} v\|_{\calK_h}^2
+h^2\|\nab_{\Gamma_h}^2v\|_{\calK_h}^2\\ 
&\qquad\qquad
+h \|v\|_{\p \calK_h}^2 
+ h \|\nab_{\Gamma_h} v\|_{\p \calK_h}^2
+h^3\|\nab_{\Gamma_h}^2v\|_{\p\calK_h}^2\lesssim \|v\|_h^2.
\end{split}
\end{equation}
Consequently, $\|\cdot\|_h$, $\tbar{\cdot}_h$, and $\tbar{\cdot}_{2h}$ are equivalent
on $V_{h,0}$.
\end{lemma}

\begin{lemma} \label{251116-1}
There holds $\tbar{w^\ell}_{*,2h} \approx \tbar{w}_{2h}$ for all $w\in W$.
Moreover, there holds $\tbar{\phi}_{*,2h}\lesssim \|\phi\|_{H^4(\Gamma)}$ for all $\phi\in H^4(\Gamma)$.
\end{lemma}
\begin{proof}
Comparing the definitions of $\tbar{\cdot}_{*, 2h}$ and $\tbar{\cdot}_{2h}$, based on \eqref{eqn:LapEquivBound2}, \eqref{eqn:GradEquivBound}, and \eqref{eqn:HessEquivBound2}, we observe that it is sufficient to focus on the term $h^{-1}\|\bmu_{\ell} \cdot [\nab_{\Gamma}w^\ell]\|_{\calE^\ell_h}^2$. By \eqref{250921-1} and \eqref{251117-4},
\begin{align*}
& h^{-1}\|\bmu_{\ell} \cdot [\nab_{\Gamma}w^\ell]\|_{\calE^\ell_h}^2 \\
& \lesssim h^{-1}\|\bmu_{\ell} \cdot [\nab_{\Gamma}w^\ell] - (\bmu_h \cdot [\nab_{\Gamma_h}w])^\ell\|_{\calE^\ell_h}^2 + h^{-1}\|(\bmu_h \cdot [\nab_{\Gamma_h}w])^\ell\|_{\calE^\ell_h}^2 \\
& \lesssim h^{-1}\|(\bmu_{\ell} \cdot [\nab_{\Gamma}w^\ell])^e - \bmu_h \cdot [\nab_{\Gamma_h}w]\|_{\calE_h}^2 + h^{-1}\|\bmu_h \cdot [\nab_{\Gamma_h}w]\|_{\calE_h}^2 \\
& \lesssim h^2\tbar{w^\ell}_{*, 2h}^2 + \tbar{w}_{2h}^2,
\end{align*}
which implies $\tbar{w^\ell}_{*, 2h}^2 \lesssim h^2\tbar{w^\ell}_{*, 2h}^2 + \tbar{w}_{2h}^2$, and therefore $\tbar{w^\ell}_{*, 2h} \lesssim \tbar{w}_{2h}$.

To prove the second assertion,
we use \eqref{eqn:GradEquivBound}, \eqref{trace-inequ-3}, \eqref{trace-inequ-1}, and \eqref{*} 
to obtain, for all $\phi\in H^4(\Gamma)$.
\begin{align}
\|\Delta_{\Gamma}\phi\|_{\calK^\ell_h}^2 + \|\nab_{\Gamma}\phi\|_{\calK^\ell_h}^2 + h^2\|\nab_{\Gamma}^2\phi\|_{\calK^\ell_h}^2 + h^{-1}\|\bmu_{\ell} \cdot [\nab_{\Gamma}\phi]\|_{\calE^\ell_h}^2
+h\|\nab_{\Gamma}\phi\|_{\partial\calK_h^\ell}^2
\lesssim \|\phi\|_{H^4(\Gamma)}^2, \label{251103-2}
\end{align}
By \eqref{eqn:HessEquivBound2}, \eqref{251103-2}, \eqref{trace-inequ-3}, \eqref{trace-inequ-1}, and \eqref{*},
\begin{align}
\begin{split} \label{251103-5}
\|\nab_{\Gamma}^2\phi\|_{\partial\calK^\ell_h}
& \lesssim \|\nab_{\Gamma_h}^2\phi^e\|_{\partial\calK_h} + h\|\nab_{\Gamma_h}\phi^e\|_{\partial\calK_h} \\
& \lesssim \|\nab^2\phi^e\|_{\partial\calK_h} + h^{\frac{1}{2}}\|\phi\|_{H^4(\Gamma)} \\
& \lesssim h^{-1}\|\nab^2\phi^e\|_{\calT_h} + \|\nab^3\phi^e\|_{\calT_h}  + h\|\nab^4\phi^e\|_{\calT_h}
 + h^{\frac{1}{2}}\|\phi\|_{H^4(\Gamma)} \\
& \lesssim h^{-\frac{1}{2}}\|\phi\|_{H^4(\Gamma)}.
\end{split}
\end{align}
Likewise, using \eqref{eqn:LapEquivBound2}, \eqref{251103-5}, and \eqref{251103-2} yields
\begin{align}
\begin{split} \label{251103-6}
\|\Delta_{\Gamma}\phi\|_{\p \calK^\ell_h}
& \lesssim \|\Delta_{\Gamma_h}\phi^e\|_{\p \calK_h} + h^2\|\nab_{\Gamma_h}^2\phi^e\|_{\partial\calK_h} + h\|\nab_{\Gamma_h}\phi^e\|_{\partial\calK_h} \\
& \lesssim \|\nab^2\phi^e\|_{\partial\calK_h} + h^{\frac{1}{2}}\|\phi\|_{H^4(\Gamma)} \\
& \lesssim h^{-\frac{1}{2}}\|\phi\|_{H^4(\Gamma)}.
\end{split}
\end{align}
Combining \eqref{251103-2}--\eqref{251103-6} completes the proof.
\end{proof}

\begin{lemma} \label{251117-1}
There holds for all $w \in W^\ell$,
\begin{align}
\|(\Delta_\Gamma w)^e - \Delta_{\Gamma_h}w^e\|_{\calK_h}
& \lesssim h\tbar{w}_{*, 2h}, \label{251117-2} \\
h^{\frac{1}{2}}\|(\{\Delta_{\Gamma}w\})^e - \{\Delta_{\Gamma_h}w^e\}\|_{\calE_h}
& \lesssim h\tbar{w}_{*, 2h}, \label{251117-3} \\
h^{-\frac{1}{2}}\big\|(\bmu_{\ell} \cdot [\nab_\Gamma w])^e - \bmu_h \cdot [\nab_{\Gamma_h}w^e]\big\|_{\calE_h}
& \lesssim h\tbar{w}_{*, 2h}. \label{251117-4}
\end{align}
\end{lemma}

\begin{proof}
The first two estimates \eqref{251117-2}--\eqref{251117-3} follow from
\eqref{LapChain} and \eqref{250921-1}. A proof of the third estimate \eqref{251117-4} is given in \cite[(C.12)--(C.14)]{neilan2025c0} and \eqref{250921-1}.
\end{proof}

\begin{lemma}\label{lem:aDiff1}
There holds for $v, w \in W$,
\begin{align}
|a_h^\ell(v^\ell, w^\ell) - a_h(v, w)|
& \lesssim h \tbar{v^\ell}_{*, 2h}\tbar{w^\ell}_{*, 2h}. \label{251117-16}
\end{align}
\end{lemma}

\begin{proof}
All terms in $a_h^\ell(v^\ell, w^\ell) - a_h(v, w)$ can be paired and rewritten as
\begin{align*}
& |a_h^\ell(v^\ell, w^\ell) - a_h(v, w)| \\
& \le \left|(\Delta_{\Gamma}v^\ell, \Delta_{\Gamma}w^\ell)_{\calK^\ell_h} - (\Delta_{\Gamma_h}v, \Delta_{\Gamma_h}w)_{\calK_h}\right| \\
& \qquad + \left|(\{\Delta_{\Gamma}v^\ell\}, \bmu_{\ell} \cdot [\nab_{\Gamma}w^\ell])_{\calE_h^\ell} - (\{\Delta_{\Gamma_h}v\}, \bmu_h \cdot [\nab_{\Gamma_h}w])_{\calE_h}\right| \\
& \qquad + \left|(\{\Delta_{\Gamma}w^\ell\}, \bmu_{\ell} \cdot [\nab_{\Gamma}v^\ell])_{\calE_h^\ell} - (\{\Delta_{\Gamma_h}w\}, \bmu_h \cdot [\nab_{\Gamma_h}v])_{\calE_h}\right| \\
& \qquad + \frac{\sigma}{h}\left|(\bmu_{\ell} \cdot [\nab_{\Gamma}v^\ell], \bmu_{\ell} \cdot [\nab_{\Gamma}w^\ell])_{\calE_h^\ell} - (\bmu_h \cdot [\nab_{\Gamma_h}v], \bmu_h \cdot [\nab_{\Gamma_h}w])_{\calE_h}\right| \\
& =: I + II + III + IV.
\end{align*}
By a change of the integration domain, adding and subtracting terms, we can rewrite $I$  as
\begin{align*}
I
& = \Big|((\mu_h - 1)(\Delta_{\Gamma}v^\ell)^e, (\Delta_{\Gamma}w^\ell)^e)_{\calK_h} \\
& \qquad + ((\Delta_{\Gamma}v^\ell)^e - \Delta_{\Gamma_h}v, (\Delta_{\Gamma}w^\ell)^e)_{\calK_h} + ((\Delta_{\Gamma}v^\ell)^e, (\Delta_{\Gamma}w^\ell)^e - \Delta_{\Gamma_h}w)_{\calK_h} \\
& \qquad - ((\Delta_{\Gamma}v^\ell)^e - \Delta_{\Gamma_h}v, (\Delta_{\Gamma}w^\ell)^e - \Delta_{\Gamma_h}w)_{\calK_h}\Big|.
\end{align*}
We then use $|\mu_h - 1| = \calO(h^2)$ and \eqref{251117-2} to obtain
\[
I
\lesssim h\tbar{v^\ell}_{*, 2h}\tbar{w^\ell}_{*, 2h}.
\]
The estimates for $II$, $III$, and $IV$ follow 
the same arguments, but using \eqref{251117-3}--\eqref{251117-4} instead of \eqref{251117-2}.
We omit the details.
\end{proof}

\begin{lemma} \label{251116-3}
There holds for all $v \in V_{h, 0}$,
\begin{align}
|l(v^\ell) - l_h(v)|
& \lesssim \|f-F^\ell_h\|_{\Gamma}\tbar{v}_{ h} \label{4.6},
\end{align}
where
\begin{equation} \label{FhDef}
F^\ell_h = \mu_h^{-1}f^\ell_h.
\end{equation}
\end{lemma}

\begin{proof}
A change of the integration domain and Lemma \ref{lem:NormEstimates}
give us
\begin{align*}
|l(v^\ell) - l_h(v)|
& = |(f-F^\ell_h, v^\ell)_\Gamma| \lesssim \|f-F^{\ell}_h\|_\Gamma \|v\|_{\Gamma_h}
\lesssim  \|f-F^{\ell}_h\|_\Gamma\tbar{v}_{h}.
\end{align*}
\end{proof}


\section{Convergence Analysis}\label{sec-converge}

\subsection{Energy Estimates}
To prove convergence of the numerical method
with respect to the $H^2$-type norm,
we first establish approximation results
of a Lagrange-type interpolant.
\begin{lemma} \label{lem:InterpFun}
Let $\pi_h:H^2(\Omega_h)\to V_{h}$ be the
quadratic Lagrange (nodal) interpolant, and define $\mathring{\pi}_h:H^2(\Omega_h)\to V_{h,0}$
such that
\[
\mathring{\pi}_h \phi = \pi_h \phi - \frac1{|\Gamma_h|} \int_{\Gamma_h} \pi_h \phi.
\]
Then there holds, for all $\phi\in H^4(\Gamma)$,
\begin{align} \label{4.8}
h \tbar{\mathring{\pi}_h \phi^e}_h+ \tbar{\phi^e - \mathring{\pi}_h \phi^e}_{2h}+ \tbar{\phi-(\mathring{\pi}_h \phi^e)^\ell}_{*,2h}\lesssim h \|\phi\|_{H^3(\Gamma)}+h^2 \|\phi\|_{H^4(\Gamma)}.
\end{align}
\end{lemma}
\begin{proof}

Let $\phi\in H^4(\Gamma)$ and note that $\phi^e\in H^4(\Omega_h)$ due to
the smoothness of $\Gamma$. 
Set $e = \phi^e - {\pi}_h \phi^e$.
By repeated use of the trace inequality \eqref{trace-inequ-1}--\eqref{trace-inequ-3} and standard approximation properties of the Lagrange interpolant, we have
\begin{align*}
\tbar{e}^2_{2h}\lesssim h^{-3}\|\nab e\|_{\calT_h}^2 + h^{-1} \|\nab^2 e\|_{\calT_h}^2 + h \|\nab^3 e\|_{\calT_h}^2 +h^3 \|\nab^4 e\|_{\calT_h}^2 \lesssim h |\phi^e|_{H^3(\Omega_h)}^2+h^3 |\phi^e|_{H^4(\Omega_h)}^2.
\end{align*}
Therefore by \eqref{*}, $\tbar{e}_{2h}\lesssim h \|\phi\|_{H^3(\Gamma)}+h^2 \|\phi\|_{H^4(\Gamma)}$.  Similar arguments,
along with \eqref{eqn:GradEquivBound}, \eqref{eqn:HessEquivBound2}, and \eqref{eqn:LapEquivBound2} yield
$\tbar{e^\ell}_{*, 2h} \lesssim h \|\phi\|_{H^3(\Gamma)}+h^2 \|\phi\|_{H^4(\Gamma)}$.

Next, we use the inverse inequality \eqref{inverse-est-3}, the local $H^2$-stability of the Lagrange interpolant, and \eqref{*} to obtain
\[
\|\Delta_{\Gamma_h} ({\pi}_h \phi^e)\|_{\calK_h}^2
\lesssim h^{-1}\|\nab^2 ({\pi}_h \phi^e)\|_{\calT_h}^2\lesssim h^{-1} \|\nab^2 \phi^e\|_{\Omega_h}^2\lesssim \|\phi\|_{H^2(\Gamma)}^2.
\]
We apply similar arguments in the norm $\|\cdot\|_h$ 
and use norm equivalence to conclude $\tbar{\pi_h \phi^e}_h\lesssim 
\|\pi_h \phi^e \|_h\lesssim \|\phi\|_{H^3(\Gamma)}$.

The estimates \eqref{4.8} now
follow from Lemma \ref{251116-1}:
\begin{align*}
h\tbar{\mathring{\pi}_h \phi^e}_h + \tbar{\phi^e-\mathring{\pi}_h \phi^e}_{2h}+\tbar{\phi-(\mathring{\pi}_h \phi^e)^\ell}_{*,2h}
&\lesssim h\tbar{\mathring{\pi}_h \phi^e}_h + \tbar{\phi^e-\mathring{\pi}_h \phi^e}_{2h}\\
&= h\tbar{{\pi}_h \phi^e}_h + \tbar{\phi^e-{\pi}_h \phi^e}_{2h}\\
&\lesssim h \|\phi\|_{H^3(\Gamma)}+h^2 \|\phi\|_{H^4(\Gamma)}.
\end{align*}
\end{proof}

\begin{theorem}[Error Estimates] \label{H2ErrorThm}
Let $u$ be the exact solution to \eqref{eqn:SBiharmonicProblem}, 
and let $u_h \in V_{h, 0}$ solve the $C^0$ IP method \eqref{disc-prob}. If $u \in H^4(\Gamma)$, then there holds
\begin{align}
\tbar{u^e - u_h}_h
& \lesssim \|f - F_h^\ell\|_{\Gamma} +h\|u\|_{H^4(\Gamma)}, \label{251019-2}
\end{align}
where $F_h^\ell$ is defined by \eqref{FhDef}.
\end{theorem}

\begin{proof}
We write for arbitrary $v\in V_{h,0}$ 
\begin{align*}
A_h(u^e,v) - l_h(v) 
&= a_h(u^e,v)-l_h(v) +s_h(u^e,v)\\
& = l(v^\ell) - l_h(v) +\big(a_h(u^e,v)-a_h^\ell(u,v^\ell)\big),
\end{align*}
where we used the identity $a_h^\ell(u,v^\ell) = l(v^\ell)$ (cf.~Remark \ref{rem:Consistent}).
Applying Lemmas \ref{251116-3}, \ref{lem:aDiff1}, \ref{251116-1} and \ref{lem:NormEstimates} yield
\begin{align*}
A_h(u^e,v) - l_h(v) 
\lesssim \|f-F_h^\ell\|_\Gamma \tbar{v}_h + h \tbar{u}_{*,2h} \tbar{v^\ell}_{*,2h}
\lesssim (\|f-F_h^\ell\|_\Gamma + h \|u\|_{H^4(\Gamma)}) \|v\|_h.
\end{align*}

It then follows from Strang's lemma  that
\begin{align*}
\tbar{u^e - u_h}_h\lesssim  \inf_{v\in V_{h,0}} \tbar{u^e-v}_{h} + \|f-F_h^\ell\|_{\Gamma} + h\|u\|_{H^4(\Gamma)}.
\end{align*}
Finally, we take $v = \mathring{\pi}_h u^e$ and apply Lemma \ref{lem:InterpFun} to obtain the result.
\end{proof}

\begin{corollary}\label{cor:2h}
There holds
\[
\tbar{u^e - u_h}_{2h}\lesssim \|f-F_h^\ell\|_\Gamma + h\|u\|_{H^4(\Gamma)}.
\]
\end{corollary}
\begin{proof}
By Lemmas \ref{lem:NormEstimates} and \ref{lem:InterpFun}, 
as well as Theorem \ref{H2ErrorThm}, we have
\begin{align*}
\tbar{u^e - u_h}_{2h}&\lesssim \tbar{u^e-\mathring{\pi}_h u^e}_{2h}+\tbar{\mathring{\pi}_h u^e-u_h}_h\\
&\lesssim 
\tbar{u^e-\mathring{\pi}_h u^e}_{2h}+\tbar{u^e - u_h}_h\lesssim \|f-F_h^\ell\|_\Gamma + h\|u\|_{H^4(\Gamma)}.
\end{align*}
\end{proof}

\subsection{$L^2$ Estimates}

In this section, we prove
that the error $u^e-u_h$ 
in the $L^2$ norm converges quadratically
with respect to the mesh parameter.
To show this result, we first
extend the result given in \cite[Lemma 3.2]{LarssonLarson17}
to the TraceFEM setting.
\begin{lemma}\label{lem:Phbn}
For $\bchi \in [W^{2, 1}(\Gamma)]^3$ and $h$ small enough, there holds
\begin{align}
\left|\int_{\Gamma_h}({\bf P}_h\bn) \cdot \bchi^e\right|
& \lesssim h^2\|\bchi\|_{W^{2, 1}(\Gamma)}, \label{250728-2} \\
\left|\int_{\Gamma_h}({\bf P}\bn_h) \cdot \bchi^e\right|
& \lesssim h^2\|\bchi\|_{W^{2, 1}(\Gamma)}. \label{250909-2}
\end{align}
\end{lemma}
\begin{proof}
First, notice that it is enough to prove \eqref{250728-2} since,
if this inequality holds, 
we can use the identity
\[
{\bf P}\bn_h
= (1 - \bn \cdot \bn_h)(\bn + \bn_h) - {\bf P}_h\bn,
\]
Holder's inequality, \eqref{eqn:GeomApprox}, and \eqref{250921-1} to get
\begin{align*}
\left|\int_{\Gamma_h}({\bf P}\bn_h) \cdot \bchi^e\right|
& = \left|\int_{\Gamma_h}(1 - \bn \cdot \bn_h)(\bn + \bn_h) \cdot \bchi^e - \int_{\Gamma_h}({\bf P}_h\bn) \cdot \bchi^e\right| \\
& \lesssim h^2\|\bchi\|_{L^1(\Gamma)} + h^2\|\bchi\|_{W^{2, 1}(\Gamma)} 
 \lesssim h^2\|\bchi\|_{W^{2, 1}(\Gamma)},
\end{align*}
i.e., \eqref{250909-2} holds. 

To begin the proof of \eqref{250728-2},
we set $\jump{\bmu_h}|_E := \bmu_{\p K}^+ + \bmu_{\p K}^-$ to denote
the jump of the co-normal across the edge $E$,
where $E = \p K^+\cap \p K^-\in \calE_h$.
By the divergence theorem, we have
\begin{equation}\label{eqn:PhnLemmaStart}
\begin{split}
\int_{\Gamma_h}({\bf P}_h\bn) \cdot \bchi^e
& = \int_{\Gamma_h}({\bf P}_h\nab d) \cdot \bchi^e \\
& = -\int_{\calK_h}d\,\nab \cdot ({\bf P}_h\bchi^e) + \int_{\calE_h}d\,\jump{\bmu_h} \cdot \bchi^e \\
& =: I_1 + I_2.
\end{split}
\end{equation}
\textit{Estimate} $I_1$:  Using that $\bn_h$ is piecewise constant, we have $ \nab \cdot ({\bf P}_h\bchi^e) = {\rm div}_{\Gamma_h} \bchi^e$.
Therefore by \eqref{eqn:GeomApprox}, Holder's inequality, and \eqref{eqn:GradEquivBound}
we have
\begin{equation}\label{eqn:PhnI1}
|I_1|
\lesssim h^2\|\nab \cdot ({\bf P}_h\bchi^e)\|_{L^1(\mathcal{K}_h)}
= h^2 \|{\rm div}_{\Gamma_h} \bchi^e\|_{L^1(\calK_h)}
\lesssim h^2\|\bchi^e\|_{W^{1, 1}(\calK_h)}
\lesssim h^2\|\bchi\|_{{W}^{1, 1}(\Gamma)}.
\end{equation}

\textit{Estimate} $I_2$: Applying Holder's inequality again yields
\begin{align}\label{eqn:I2FirstBound}
|I_2|
& \lesssim \|d\|_{L^\infty(\calE_h)}\|\jump{\bmu_h}\|_{L^\infty(\calE_h)}\|\bchi^e\|_{L^1(\p\calK_h)}\lesssim h^3 \|\bchi^e\|_{L^1(\p \calK_h)},
\end{align}
where we used
\begin{align*}
\|d\|_{L^\infty(\calE_h)}
& \lesssim \|d\|_{L^\infty(\Gamma_h)} \lesssim h^2, \\
\|\jump{\bmu_h}\|_{L^\infty(\calE_h)}
& \lesssim \|\bmu_{\p K}^+ - \bmu_{\p K^\ell}^+\|_{L^\infty(\calE_h)} + \|\bmu_{\p K}^- - \bmu_{\p K^\ell}^-\|_{L^\infty(\calE_h)} \lesssim h.
\end{align*}
Thanks to \eqref{trace-inequ-3}, \eqref{trace-inequ-1}, and \eqref{250728-1}, we have
\begin{align*}
\|\bchi^e\|_{L^1(\p\calK_h)}
& \lesssim h^{-1}\|\bchi^e\|_{L^1(\calF_h)} + \|\nab\bchi^e\|_{L^1(\calF_h)} \\
& \lesssim h^{-2}\|\bchi^e\|_{L^1(\calT_h)} + h^{-1}\|\nab\bchi^e\|_{L^1(\calT_h)} + \|\nab^2\bchi^e\|_{L^1(\calT_h)} \\
& \lesssim h^{-1}\|\bchi\|_{L^1(\Gamma)} + \|\bchi\|_{W^{1, 1}(\Gamma)} + h\|\bchi\|_{W^{2, 1}(\Gamma)} \\
& \lesssim h^{-1}\|\bchi\|_{W^{2, 1}(\Gamma)}.
\end{align*}
Applying this estimate to \eqref{eqn:I2FirstBound}
yields  $|I_2| \lesssim h^2\|\bchi\|_{W^{2, 1}(\Gamma)}$.
Combining this inequality with \eqref{eqn:PhnLemmaStart}-- \eqref{eqn:PhnI1}
yields \eqref{250728-2}, which completes the proof.
\end{proof}

Using Lemma \ref{lem:Phbn},
we extend the geometric
consistency estimate \eqref{251117-16}
for functions with higher regularity.

\begin{lemma} \label{250902-2}
For $\psi \in H^4(\Gamma), \phi \in H^3(\Gamma)$, the following integral estimates hold:
\begin{align}
((\Delta_\Gamma\psi)^e, (\Delta_\Gamma\phi)^e - \Delta_{\Gamma_h}\phi^e)_{\calK_h}
& \lesssim h^2\|\psi\|_{H^4(\Gamma)}\|\phi\|_{H^3(\Gamma)}, \label{250728-3} \\
((\{\Delta_\Gamma\psi\})^e, (\bmu_\ell \cdot [\nab_\Gamma\phi])^e - \bmu_h \cdot [\nab_{\Gamma_h}\phi^e])_{\calE_h}
& \lesssim h^2\|\psi\|_{H^4(\Gamma)}\|\phi\|_{H^3(\Gamma)}. \label{250728-4}
\end{align}
\end{lemma}

\begin{proof}
\textit{Estimate} \eqref{250728-3}: Recalling \eqref{LapChain}, we have
\begin{align*}
& ((\Delta_\Gamma\psi)^e, (\Delta_\Gamma\phi)^e - \Delta_{\Gamma_h}\phi^e)_{\calK_h} \\
& \lesssim h^2|\psi|_{H^2(\Gamma)}(|\phi|_{H^1(\Gamma)} + |\phi|_{H^2(\Gamma)}) + ((\Delta_\Gamma\psi)^e, (\bn \cdot \bn_h)\bn_h \cdot {\bf H}(\nab_\Gamma\phi)^e)_{\calK_h}.
\end{align*}
Applying \eqref{250728-2} with $\bchi = (\Delta_\Gamma\psi){\bf H}(\nab_\Gamma\phi)$ to the second term leads to
\begin{align*}
& ((\Delta_\Gamma\psi)^e, (\bn \cdot \bn_h)\bn_h \cdot {\bf H}(\nab_\Gamma\phi)^e)_{\calK_h} \\
& = -((\Delta_\Gamma\psi)^e, (\bn - (\bn \cdot \bn_h)\bn_h) \cdot {\bf H}(\nab_\Gamma\phi)^e)_{\calK_h} \\
& = -({\bf P}_n\bn, (\Delta_\Gamma\psi)^e{\bf H}(\nab_\Gamma\phi)^e)_{\calK_h} \\
& \lesssim h^2\|\Delta_\Gamma\psi{\bf H}\nab_\Gamma\phi\|_{W^{2, 1}(\Gamma)} \\
%
& \lesssim h^2\|\psi\|_{H^4(\Gamma)}\|\phi\|_{H^3(\Gamma)},
\end{align*}
which concludes the proof of \eqref{250728-3}. \\
\textit{Estimate} \eqref{250728-4}: The Sobolev embedding $H^2(\Gamma) \hookrightarrow L^\infty(\Gamma)$, and the assumed regularity of $\psi$ and $\chi$, imply $\Delta_\Gamma\psi$ and $\nab_\Gamma\chi$ are continuous, i.e., $\{\Delta\psi\} = \Delta\psi$ and $[\nab_\Gamma\phi] = 0$. Thus, due to \eqref{250914-1} we have
\begin{align}
\begin{split} \label{250830-1}
& ((\{\Delta_\Gamma\psi\})^e, (\bmu_\ell \cdot [\nab_\Gamma\phi])^e - \bmu_h \cdot [\nab_{\Gamma_h}\phi^e])_{\calE_h} \\
& = -((\Delta_\Gamma\psi)^e, \bmu_h \cdot [{\bf P}_h({\bf P} - d{\bf H})(\nab_\Gamma\phi)^e])_{\calE_h} \\
& = ((\Delta_\Gamma\psi)^e, \bmu_h \cdot [{\bf P}_hd{\bf H}(\nab_\Gamma\phi)^e])_{\calE_h} - ((\Delta_\Gamma\psi)^e, \bmu_h \cdot [{\bf P}_h(\nab_\Gamma\phi)^e])_{\calE_h} \\
& =: I_1 + I_2,
\end{split}
\end{align}
which we will estimate separately. \\
\textit{Estimate} $I_1$: By \eqref{trace-inequ-3}, \eqref{trace-inequ-1}, \eqref{250728-1}, and Holder's inequality, we have
\begin{align*}
 \|(\Delta_\Gamma\psi)^e\|_{L^1(\calE_h)} 
& \lesssim h^{-2}\|(\Delta_\Gamma\psi)^e\|_{L^1(\calT_h)} + h^{-1}\|\nab(\Delta_\Gamma\psi)^e\|_{L^1(\calT_h)} + \|\nab^2(\Delta_\Gamma\psi)^e\|_{L^1(\calT_h)} \\
& \lesssim h^{-1}\|\Delta_\Gamma\psi\|_{L^1(\Gamma)} + \|\Delta_\Gamma\psi\|_{W^{1, 1}(\Gamma)} + h\|\Delta_\Gamma\psi\|_{W^{2, 1}(\Gamma)} \\
& \lesssim h^{-1}\|\psi\|_{H^4(\Gamma)},
\intertext{and}
 \|(\nab_\Gamma\phi)^e\|_{L^1(\calE_h)}
& \lesssim h^{-2}\|(\nab_\Gamma\phi)^e\|_{L^1(\calT_h)} + h^{-1}\|\nab(\nab_\Gamma\phi)^e\|_{L^1(\calT_h)} + \|\nab^2(\nab_\Gamma\phi)^e\|_{L^1(\calT_h)} \\
& \lesssim h^{-1}\|\nab_\Gamma\phi\|_{L^1(\Gamma)} + \|\nab_\Gamma\phi\|_{W^{1, 1}(\Gamma)} + h\|\nab_\Gamma\phi\|_{W^{2, 1}(\Gamma)} \\
& \lesssim h^{-1}\|\phi\|_{H^3(\Gamma)}.
\end{align*}
Noting that (cf. \cite[(3.9)]{neilan2025c0})
\[
\|{\bf P}\bmu_h - \bmu_\ell\|_{L^\infty(\p K)} \lesssim h^2,
\]
and so by using
\[
{\bf P}(\bmu_{\p K}^+ + \bmu_{\p K}^-)
= ({\bf P}\bmu_{\p K}^+ - \bmu_{\p K^\ell}^+) + ({\bf P}\bmu_{\p K}^- - \bmu_{\p K^\ell}^-),
\]
\eqref{eqn:GeomApprox}, and the boundedness of ${\bf H}$, it follows that
\begin{align}
\begin{split} \label{250830-2}
I_1
& = ((\Delta_\Gamma\psi)^e, (\bmu_{\p K}^+ + \bmu_{\p K}^-) \cdot (d{\bf H}(\nab_\Gamma\phi)^e))_{\calE_h} \\
& \lesssim h^2(\|{\bf P}\bmu_{\p K}^+ - \bmu_{\p K^\ell}^+\|_{L^\infty(\calE_h)} + \|{\bf P}\bmu_{\p K}^- - \bmu_{\p K^\ell}^-\|_{L^\infty(\calE_h)})\|(\Delta_\Gamma\psi)^e\|_{L^1(\calE_h)}\|(\nab_\Gamma\phi)^e\|_{L^1(\calE_h)} \\
& \lesssim h^2\|\psi\|_{H^4(\Gamma)}\|\phi\|_{H^3(\Gamma)}.
\end{split}
\end{align}
\textit{Estimate} $I_2$: Before estimating this term, we set $\bchi := \Delta_\Gamma \psi \nab_\Gamma \phi$ and observe that $\bchi$ is tangent to $\Gamma$. Thus, $\bchi^e \cdot \bn = 0$ in a neighborhood of $\Gamma$. Consequently, the product rule and the identity ${\bf H} = \Grad\bn$ lead to $0 = \nab(\bchi^e \cdot \bn) = (\Grad\bchi^e)^\intercal\bn + {\bf H}\bchi^e$, which implies
\begin{align}
(\Grad \bchi^e)^\intercal\bn = -{\bf H}\bchi^e. \label{250909-1}
\end{align}
Based on \cite[(C.33)--(C.42)]{LarssonLarson17},
\begin{align}
\begin{split} \label{250830-3}
I_2
& = (({\rm div}_\Gamma\bchi)^e, 1)_{\calK_h} - d({\rm tr}({\bf Z}), 1)_{\calK_h} - ({\bf P}\bn_h, (\Grad\bchi^e)^\intercal\bn_h)_{\calK_h} \\
& =: I_{2,1} + I_{2,2} + I_{2,3},
\end{split}
\end{align}
where ${\bf Z}
= (\Grad_\Gamma\bchi{\bf H})^e + (\nab{\bf P}\bchi{\bf H})^e$, and $\nab{\bf P}\bchi$ is a $3 \times 3$ matrix with entries similar to \eqref{250909-3}. Since $II_1, II_2$ satisfy (cf. \cite[(C.43)--(C.46)]{LarssonLarson17})
\begin{align}
I_{2,1} \lesssim h^2\|\psi\|_{H^3(\Gamma)}\|\phi\|_{H^2(\Gamma)}, \qquad I_{2,2} \lesssim h^2\|\psi\|_{H^3(\Gamma)}\|\phi\|_{H^2(\Gamma)}, \label{250830-4}
\end{align}
we just need to focus on the term $I_{2,3}$. By virtue of \eqref{250909-1}, Holder’s inequality, \eqref{250909-2}, \eqref{trace-inequ-2}, and \eqref{250728-1},
\begin{align*}
I_{2,3}
& = - ({\bf P}\bn_h, (\Grad\bchi^e)^\intercal(\bn_h - \bn))_{\calK_h} - ({\bf P}\bn_h, (\Grad\bchi^e)^\intercal\bn)_{\calK_h} \\
& = - ({\bf P}\bn_h, (\Grad\bchi^e)^\intercal(\bn_h - \bn))_{\calK_h} + ({\bf P}\bn_h, {\bf H}\bchi^e)_{\calK_h} \\
& \lesssim \|{\bf P}\bn_h\|_{L^\infty(\Gamma_h)}\|\Grad\bchi^e\|_{L^1(\Gamma_h)}\|\bn_h - \bn\|_{L^\infty(\Gamma_h)} + h^2\|{\bf H}\bchi\|_{W^{2, 1}(\Gamma)} \\
& \lesssim h^2\|\Grad\bchi^e\|_{L^1(\Gamma_h)} + h^2\|\psi\|_{H^4(\Gamma)}\|\phi\|_{H^3(\Gamma)} \\
& \lesssim h^2(h^{-1}\|\Grad\bchi^e\|_{L^1(\calT_h)} + \|\Grad^2\bchi^e\|_{L^1(\calT_h)}) + h^2\|\psi\|_{H^4(\Gamma)}\|\phi\|_{H^3(\Gamma)} \\
& \lesssim h^2\|\bchi\|_{W^{1, 1}(\Gamma)} + h^3\|\bchi\|_{W^{2, 1}(\Gamma)} + h^2\|\psi\|_{H^4(\Gamma)}\|\phi\|_{H^3(\Gamma)} \\
& \lesssim h^2\|\psi\|_{H^4(\Gamma)}\|\phi\|_{H^3(\Gamma)}.
\end{align*}
Combining the above inequality with \eqref{250830-1}--\eqref{250830-4} completes the proof.
\end{proof}






\begin{lemma}
There holds for $v, w \in H^4(\Gamma)$,
\begin{align}
|a_h^\ell(v, w) - a_h(v^e, w^e)|
& \lesssim h^2\|v\|_{H^4(\Gamma)}\|w\|_{H^4(\Gamma)}. \label{250902-4}
\end{align}
\end{lemma}

\begin{proof}
All terms in $a_h^\ell(v, w) - a_h(v^e, w^e)$ can be paired and rewritten as
\begin{align*}
& |a_h^\ell(v, w) - a_h(v^e, w^e)| \\
& \le \left|(\Delta_{\Gamma}v, \Delta_{\Gamma}w)_{\calK^\ell_h} - (\Delta_{\Gamma_h}v^e, \Delta_{\Gamma_h}w^e)_{\calK_h}\right| \\
& \qquad + \left|(\{\Delta_{\Gamma}v\}, \bmu_{\ell} \cdot [\nab_{\Gamma}w])_{\calE_h^\ell} - (\{\Delta_{\Gamma_h}v^e\}, \bmu_h \cdot [\nab_{\Gamma_h}w^e])_{\calE_h}\right| \\
& \qquad + \left|(\{\Delta_{\Gamma}w\}, \bmu_{\ell} \cdot [\nab_{\Gamma}v])_{\calE_h^\ell} - (\{\Delta_{\Gamma_h}w^e\}, \bmu_h \cdot [\nab_{\Gamma_h}v^e])_{\calE_h}\right| \\
& \quad + \frac{\sigma}{h}\left|(\bmu_{\ell} \cdot [\nab_{\Gamma}v], \bmu_{\ell} \cdot [\nab_{\Gamma}w])_{\calE_h^\ell} - (\bmu_h \cdot [\nab_{\Gamma_h}v^e], \bmu_h \cdot [\nab_{\Gamma_h}w^e])_{\calE_h}\right| \\
& =: I_1 + I_2 + I_3 + I_4.
\end{align*}
By a change of the integration domain, adding and subtracting terms, we can rewrite $I_1$  as
\begin{align*}
I_1
& \lesssim \Big|((\mu_h - 1)(\Delta_{\Gamma}v)^e, (\Delta_{\Gamma}w)^e)_{\calK_h}\Big| \\
& \qquad + \Big|((\Delta_{\Gamma}v)^e - \Delta_{\Gamma_h}v^e, (\Delta_{\Gamma}w)^e)_{\calK_h}\Big| + \Big|((\Delta_{\Gamma}v)^e, (\Delta_{\Gamma}w)^e - \Delta_{\Gamma_h}w^e)_{\calK_h}\Big| \\
& \qquad + \Big| - ((\Delta_{\Gamma}v)^e - \Delta_{\Gamma_h}v^e, (\Delta_{\Gamma}w)^e - \Delta_{\Gamma_h}w^e)_{\calK_h}\Big|.
\end{align*}
We then use $|\mu_h - 1| = \calO(h^2)$, Lemma \ref{250902-2}, Lemma \ref{251117-1}, and Lemma \ref{251116-1} to obtain
\[
I_1
\lesssim h^2\|v\|_{H^4(\Gamma)}\|w\|_{H^4(\Gamma)}.
\]
Note that the Sobolev embedding $H^3(\Gamma) \hookrightarrow C^1(\Gamma)$ implies $(\bmu_\ell \cdot [\nab_\Gamma v])^e = 0$, and therefore the estimates for $I_2, I_3$, and $I_4$ follow similar arguments.
\end{proof}

\begin{lemma} \label{lem:GDiffer}
Let $\phi \in H^4(\Gamma)$ and let $\mathring{\pi}_h\phi^e \in V_{h, 0}$ be its interpolant given in Lemma \ref{lem:InterpFun}. Then there holds
\begin{equation} \label{eqn:GDiffABC}
\big|a_h(u_h, \mathring{\pi}_h\phi^e) - a_h^\ell(u_h^\ell, (\mathring{\pi}_h\phi^e)^\ell)\big|
\lesssim \big(h^2\|f_h\|_{\Gamma_h} + h\|f - F_h^\ell\|_\Gamma + h^2\|u\|_{H^4(\Gamma)}\big)\|\phi\|_{H^4(\Gamma)}.
\end{equation}
\end{lemma}

\begin{proof}
Adding and subtracting terms yield
\begin{align*}
a_h(u_h, \mathring{\pi}_h\phi^e) - a_h^\ell(u_h^\ell, (\mathring{\pi}_h\phi^e)^\ell)
& = \big(a_h(u_h, \mathring{\pi}_h\phi^e - \phi^e) - a_h^\ell(u_h^\ell, (\mathring{\pi}_h\phi^e)^\ell - \phi)\big) \\
& \qquad + \big(a_h(u_h - u^e, \phi^e) - a_h^\ell(u_h^\ell - u, \phi)\big) \\
& \qquad + \big(a_h(u^e, \phi^e) - a_h^\ell(u, \phi)\big) \\
& =: I_1 + I_2 + I_3.
\end{align*}

Note that by the coercivity of the bilinear form $A_h(\cdot,\cdot)$
and Lemma \ref{lem:NormEstimates}, there holds $\tbar{u_h}_{2h}\lesssim \|u_h\|_h\lesssim \|f_h\|_{\Gamma_h}$.
It follows from \eqref{251117-16}, Lemmas \ref{lem:NormEstimates}--\ref{251116-1}, Lemma \ref{lem:InterpFun},  Corollary \ref{cor:2h} 
and \eqref{250902-4} that
\begin{align*}
I_1
& \lesssim h\tbar{u_h^\ell}_{*, 2h}\tbar{(\mathring{\pi}_h\phi^e)^\ell - \phi}_{*, 2h} 
 \lesssim h\tbar{u_h}_{2h}\tbar{\mathring{\pi}_h\phi^e - \phi^e}_{2h} 
 \lesssim h^2\|f_h\|_{\Gamma_h}\|\phi\|_{H^4(\Gamma)}, \\
I_2
& \lesssim h\tbar{u_h^\ell - u}_{*, 2h}\tbar{\phi}_{*, 2h} 
 \lesssim h\tbar{u_h - u^e}_{2h}\|\phi\|_{H^4(\Gamma)} 
 \lesssim (h\|f - F_h^\ell\|_{\Gamma} + h^2\|u\|_{H^4(\Gamma)})\|\phi\|_{H^4(\Gamma)}, \\
I_3
& \lesssim h^2\|u\|_{H^4(\Gamma)}\|\phi\|_{H^4(\Gamma)},
\end{align*}
which completes the proof.
\end{proof}

\begin{theorem}\label{thm:L2}
Let $u$ be the exact solution, and let $u_h \in V_{h, 0}$ solve the $C^0$ IP method \eqref{disc-prob}. If $u \in H^4(\Gamma)$, then there holds
\begin{align*}
\|u - u_h^\ell\|_{\Gamma}
\lesssim \|f - F^\ell_h\|_{\Gamma} + h^2\|u\|_{H^4(\Gamma)} + h^2\|f_h\|_{\Gamma_h}. 
\end{align*}
\end{theorem}

\begin{proof}
To prove the estimate we apply a duality argument. For $\psi \in C^0(\Gamma)$ with $\int_\Gamma \psi =0$, let $\phi \in H^2(\Gamma)$ satisfy $\Delta^2_\Gamma\phi = \psi$ and $\int_\Gamma \phi = 0$. By elliptic regularity, $\phi \in H^4(\Gamma)$ with
\begin{align}
\|\phi\|_{H^4(\Gamma)}
& \lesssim \|\psi\|_\Gamma. \label{251117-9}
\end{align}
Furthermore, by consistency of the $C^0$ IP method,
\begin{align}
a_h^\ell(v, \phi)
= (v, \psi)_\Gamma \qquad \forall v\in W^\ell 
\label{251117-10}
\end{align}
We decompose the error as
\begin{align}
\begin{split} \label{251117-11}
u - u_h^\ell
& = \Bigg(u - \left(u_h^\ell - \frac{1}{|\Gamma|}\int_\Gamma u_h^\ell\right)\Bigg) + \Bigg(\frac{1}{|\Gamma_h|}\int_{\Gamma_h}u_h - \frac{1}{|\Gamma|}\int_\Gamma u_h^\ell\Bigg) \\
& =: I_1 + I_2.
\end{split}
\end{align}
Consider $v = \psi = I_1$ in \eqref{251117-10} and note that $\int_\Gamma u_h^\ell / |\Gamma|$ is a constant. Therefore,
\begin{align}
\begin{split} \label{251117-12}
\|I_1\|_\Gamma^2
& = a_h^\ell(I_1, \phi) \\
& = a_h^\ell(u - u_h^\ell, \phi - (\mathring{\pi}_h\phi^e)^\ell) + a_h^\ell(u - u_h^\ell, (\mathring{\pi}_h\phi^e)^\ell) \\
& =: I_{1,1} + I_{1,2}.
\end{split}
\end{align}
Due to the Cauchy-Schwarz inequality, Lemma \ref{251116-1}, 
Lemma \ref{lem:InterpFun}, and Corollary \ref{cor:2h},
\begin{align}
\begin{split} \label{251117-13}
I_{1,1}
& \lesssim \tbar{u - u_h^\ell}_{*, 2h}\tbar{\phi - (\mathring{\pi}_h\phi^e)^\ell}_{*, 2h} \\
& \lesssim \tbar{u^e - u_h}_{2h}\tbar{\phi^e - \mathring{\pi}_h\phi^e}_{2h} \\
&\lesssim (h\|f-F_h^\ell\|_\Gamma+h^2 \|u\|_{H^4(\Gamma)}) \|I_1\|_\Gamma.
\end{split}
\end{align}
By adding and subtracting terms, we rewrite $I_{1,2}$ as
\begin{align}
\begin{split} \label{251117-14}
I_{1,2}
& = a_h^\ell(u, (\mathring{\pi}_h\phi^e)^\ell) - a_h^\ell(u_h^\ell, (\mathring{\pi}_h\phi^e)^\ell) \\
& = l((\mathring{\pi}_h\phi^e)^\ell) - a_h^\ell(u_h^\ell, (\mathring{\pi}_h\phi^e)^\ell) \\
& = l((\mathring{\pi}_h\phi^e)^\ell) - l_h(\mathring{\pi}_h\phi^e) + a_h(u_h, \mathring{\pi}_h\phi^e) - a_h^\ell(u_h^\ell, (\mathring{\pi}_h\phi^e)^\ell) \\
& \qquad + s_h(u_h, \mathring{\pi}_h\phi^e) \\
& = \Big(l((\mathring{\pi}_h\phi^e)^\ell) - l_h(\mathring{\pi}_h\phi^e)\Big) + \Big(a_h(u_h, \mathring{\pi}_h\phi^e) - a_h^\ell(u_h^\ell, (\mathring{\pi}_h\phi^e)^\ell)\Big) \\
& \qquad + s_h(u_h - u^e, \mathring{\pi}_h\phi^e - \phi^e) \\
& =: I_{1,2,1} + I_{1,2,2} + I_{1,2,3}.
\end{split}
\end{align}
It follows from \eqref{4.6}, \eqref{4.8}, and \eqref{251117-9} that
\begin{align}
\begin{split} \label{251117-15}
I_{1,2,1}
\lesssim \|f - F_h^\ell\|_\Gamma\tbar{\mathring{\pi}_h\phi^e}_h
\lesssim \|f - F_h^\ell\|_\Gamma\|\phi\|_{H^4(\Gamma)}
\lesssim \|f - F_h^\ell\|_\Gamma\|I_1\|_\Gamma.
\end{split}
\end{align}
Likewise, by \eqref{eqn:GDiffABC} and \eqref{251117-9}, we have
\begin{align}
I_{1,2,2}
& \lesssim (h\|f - F_h^\ell\|_{\Gamma} + h^2\|u\|_{H^4(\Gamma)})\|I_1\|_\Gamma. \label{251117-17}
\end{align}
Thanks to Cauchy-Schwarz inequality, \eqref{251019-2}, \eqref{4.8}, and \eqref{251117-9},
\begin{align*}
I_{1,2,3}
& \lesssim \tbar{u_h - u^e}_h\tbar{\mathring{\pi}_h\phi^e - \phi^e}_h 
 \lesssim (h\|f-F_h^\ell\|_{\Gamma} + h^2\|u\|_{H^4(\Gamma)})\|I_1\|_\Gamma.
\end{align*}
Combining \eqref{251117-14}--\eqref{251117-17} leads to
\begin{align}
I_{1,2}
& \lesssim (\|f-F_h^\ell\|_{\Gamma} + h^2\|u\|_{H^4(\Gamma)})\|I_1\|_\Gamma, \label{251117-18}
\end{align}
and so by  \eqref{251117-12}, \eqref{251117-13}, and \eqref{251117-18},
we have
\begin{align}
\|I_1\|_\Gamma
& \lesssim  \|f-F^\ell_h\|_{\Gamma} + h^2\|u\|_{H^4(\Gamma)}. \label{251117-19}
\end{align}

Next, the coercivity of $A_h(\cdot,\cdot)$ \eqref{3.11},
the Cauchy-Schwarz inequality, and \eqref{eqn:AllAtOnce} 
yield $\|u_h\|_h\lesssim \|f_h\|_{\Gamma_h}$.
Therefore by \cite[(6.39)]{burman2017cut} and \eqref{eqn:AllAtOnce}, we have
\begin{align}
\|I_2\|_\Gamma
\lesssim \left|\frac{1}{|\Gamma_h|}\int_{\Gamma_h}u_h - \frac{1}{|\Gamma|}\int_\Gamma u_h^\ell\right|
\lesssim h^2\|u_h\|_{\calK_h}
\lesssim h^2\|u_h\|_h
\lesssim h^2\|f_h\|_{\Gamma_h}. \label{251117-21}
\end{align}
Based on \eqref{251117-11}, \eqref{251117-19}, and \eqref{251117-21}, we conclude that
\[
\|u - u_h^\ell\|_\Gamma
\lesssim \|f-F^\ell_h\|_{\Gamma} + h^2\|u\|_{H^4(\Gamma)} + h^2\|f_h\|_{\Gamma_h}.
\]
\end{proof}

\section{Numerical Examples}\label{sec-numerics}
In this section, we perform 
some simple numerical examples to gauge
the theoretical results and determine whether the derived
error estimates are sharp. In the set of numerical examples,
we take $\Gamma$ to be the unit sphere, and take the data $f$
such that the exact solution is given by $u = \exp(x+y^2)\cos(z^3)$.
The resulting expression for $f$ is well-defined on $\bbR^3$,
and we use this expression in the form $l_h(\cdot)$.
In this case, simple arguments show $\|f-F_h^\ell\|_{\Gamma}\lesssim h^2$.
The set of numerical experiments are implemented using
Netgen/NGsolve \cite{ngsolve} with the add-on ngsxfem \cite{LHPvW21}, which supports unfitted
finite element discretizations.

In addition, to the $C^0$ IP method \eqref{disc-prob},
we solve the method with different
bilinear forms $s_h(\cdot,\cdot)$
to determine the minimal about of TraceFEM stabilizations
to ensure well-posedness and convergence of the scheme.
In particular, we consider three $C^0$ IP methods: Find $u_h\in V_{h,0}$ such that
\begin{align}\label{eqn:General-C0IP}
A^{(j)}_h(u_h, v)
& := a_h(u_h, w) + s_h^{(j)}(u_h, v)= l_h(v)\qquad \forall v\in V_{h,0},\quad j=0,1,2,
\end{align}
where $a_h(\cdot,\cdot)$ is given by \eqref{eqn:ahDef}, $s_h^{(0)}(\cdot,\cdot) = s_h(\cdot,\cdot)$
given by \eqref{eqn:shDef} (so that the case $j=0$ yields the scheme \eqref{disc-prob}), and 
\begin{align}
s_h^{(1)}(v, w)
& = \frac{\beta}{h^2}([\nab v], [\nab w])_{\calF_h} + \gamma([\nab^2v], [\nab^2w])_{\calF_h}, \label{5.1} \\
s_h^{(2)}(v, w)
& = \gamma
([\nab^2v], [\nab^2w])_{\calF_h}, \label{5.2}
\end{align}
with $\gamma,\beta>0$. In all numerical experiments, we take $\gamma = \beta = \sigma=10$.

We report the number of degrees of freedom (ndof), $L^2$, $H^1$, and Laplace-Beltrami $L^2$
errors in Tables \ref{tab:MethodJ0}--\ref{tab:MethodJ2}
for a sequence of uniform mesh refinements.
In all three cases, we observe
similar convergence rates across the different norms.
In particular,
both the $L^2$ errors and $H^1$ errors exhibit rates
approaching linear convergence 
with respect to the number of unknowns, which corresponds
to a quadratic rate of convergence with respect to the mesh parameter $h$.
For $j=0$, the observed $L^2$ convergence rates are consistent with
the theoretical results
given in Theorem \ref{thm:L2}. The results in Table \ref{tab:MethodJ2} (the case $j=2$) 
indicate that including the gradient jump in the stabilization $s_h(\cdot,\cdot)$
is not essential to achieve converage of the method.

The convergence behavior of the Laplace-Beltrami $L^2$ error
$\|\Delta_{\Gamma_h}(u^e - u_h)\|_{\calK_h}$ reported in
Tables \ref{tab:MethodJ0}--\ref{tab:MethodJ2}
is less clear. In the case $j=0$, Theorem \ref{H2ErrorThm} predicts
linear convergence of this term with respect to the mesh parameter 
or $\|\Delta_{\Gamma_h}(u^e - u_h)\|_{\calK_h} = O({\rm ndofs}^{-\frac12})$.
However, for all three variants of the $C^0$ IP method
we observe better rates of convergence than those guaranteed
by the theory.


\begin{table}[h]
\centering
\caption{\label{tab:MethodJ0}Errors and rates of convergence for method \eqref{disc-prob},
equivalent to method \eqref{eqn:General-C0IP} with $j=0$.}
\begin{tabular}{c c c c c c c}
\hline
ndof & $\|u^e - u_{h}\|_{\Gamma_{h}}$ & Rate  & $\|\nab_{\Gamma_h}(u^e - u_h)\|_{\Gamma_h}$ & Rate  & $\|\Delta_{\Gamma_h}(u^e - u_h)\|_{\calK_h}$ & Rate \\
\hline
3995	&1.57792	&	    &3.06759	&	     &11.45985	&\\
16802	&0.91473	&-0.380	&1.87263	&-0.344	&7.56334	&-0.289\\
67351	&0.33155	&-0.731	&0.72784	&-0.681	&3.44823	&-0.566\\
269697	&0.09524	&-0.899	&0.21967	&-0.863	&1.24509	&-0.734\\
1079322	&0.02513	&-0.960	&0.05973	&-0.939	&0.42899	&-0.768
\end{tabular}
\end{table}

\begin{table}[h]
\centering
\caption{\label{tab:MethodJ1}Errors and rates of convergence for method \eqref{eqn:General-C0IP} with $j=1$. }
\begin{tabular}{c c c c c c c}
\hline
 ndof & $\|u^e - u_{h}\|_{\Gamma_{h}}$ & Rate  &$\|\nab_{\Gamma_h}(u^e - u_h)\|_{\Gamma_h}$ & Rate &$\|\Delta_{\Gamma_h}(u^e - u_h)\|_{\calK_h}$ & Rate \\
\hline
3995	&2.0102	&	     &3.9875	&	&12.116	&\\
16802	&1.2615	&-0.324	&2.7521	&-0.258	&8.4952	&-0.247\\
67351	&0.5110	&-0.651	&1.2107	&-0.591	&4.0870	&-0.527\\
269697	&0.1552	&-0.859	&0.3869	&-0.822	&1.4776	&-0.733\\
1079322	&0.0418	&-0.945	&0.1069	&-0.927	&0.4931	&-0.791
\end{tabular}
\end{table}

\begin{table}[h]
\centering
\caption{\label{tab:MethodJ2}Errors and rates of convergence for method \eqref{eqn:General-C0IP} with $j=2$. }
\begin{tabular}{c c c c c c c }
\hline
ndof & $\|u^e - u_{h}\|_{\Gamma_{h}}$ & Rate & $\|\nab_{\Gamma_h}(u^e - u_h)\|_{\Gamma_h}$ & Rate & $\|\Delta_{\Gamma_h}(u^e - u_h)\|_{\calK_h}$ & Rate\\
\hline
3995	&1.37164	&	        &2.72862	&	     &11.02039	&\\
16802	&0.73762	&-0.432	     &1.57798	&-0.381	&7.03947	&-0.312\\
67351	&0.2619	      &-0.746	&0.61151	&-0.683	&3.20659	&-0.566\\
269697	&0.07499	&-0.901	&0.18621	&-0.857	&1.17518	&-0.724\\
1079322	&0.0198	    &-0.958	    &0.05112	&-0.932	&0.41288	&-0.754
\end{tabular}
\end{table}

\bibliographystyle{siam}
\bibliography{ref}

\appendix

\section{Proof of Lemma \ref{lem:NormEstimates}}\label{sec:ProofOfNormEstimates}

{\bf Step 1:}
Let ${\bm c} = \lambda_{\Gamma_h}(\nab v)$ and write
\begin{align*}
\|\nab v\|_{\calT_h}^2\lesssim \|\nab v - {\bm c}\|_{\calT_h}^2 +\|{\bm c}\|_{\calT_h}^2.
\end{align*}
Applying the same arguments as
in the proof of Lemma \ref{lem:72} (also see \cite[Lemma 5.3]{burman2017cut}) yields
\begin{align*}
\|{\bm c}\|_{\calT_h}^2 
&\lesssim h \|{\bm c}\|_{\calK_h}^2\lesssim h \|{\bm P_h}{\bm c}\|_{\calK_h}^2
\lesssim h\|\nab_{\Gamma_h} v\|_{\calK_h}^2 + h\|{\bm P_h}{\bm c}-\nab_{\Gamma_h}v\|_{\calK_h}^2
\lesssim h\|\nab_{\Gamma_h} v\|_{\calK_h}^2 + \|{\bm c}-\nab v\|_{\calT_h}^2.
\end{align*}
Thus we find, using \eqref{7.1} and the inverse inequality \eqref{inverse-est-3},
\begin{align}
\begin{split} \label{251110-7}
\|\nab v\|_{\calT_h}^2
& \lesssim \|\nab v-{\bm c}\|_{\calT_h}^2 + h \|\nab_{\Gamma_h} v\|_{\calK_h}^2 \\
& \lesssim h \|\Grad_{\Gamma_h} \nab v\|_{\calK_h}^2 + h^2 \|\nab^2 v\|_{\calT_h}^2 + h^{-1}\|[\nab v]\|_{\calF_h}^2+ h \|\nab_{\Gamma_h} v\|_{\calK_h}^2 \\
& \lesssim \|\nab^2 v\|_{\calT_h}^2 + h^{-1}\|[\nab v]\|_{\calF_h}^2+ h \|\nab_{\Gamma_h} v\|_{\calK_h}^2.
\end{split}
\end{align}
We then conclude from \eqref{lem-7.2} that
 \begin{align}
h\|\nab v\|_{\calT_h}^2
 & \lesssim h^2 \|\Delta_{\Gamma_h} v\|_{\calK_h}^2 + \|[\nab^2 v]\|_{\calF_h}^2+ 
 \|[\nab v]\|_{\calF_h}^2+ h^2 \|\nab_{\Gamma_h} v\|_{\calK_h}^2
 \lesssim \|v\|_{h}^2 + h^2 \|\nab_{\Gamma_h} v\|_{\calK_h}^2.
 \label{251110-5}
 \end{align}

{\bf Step 2:}
We apply \eqref{inverse-est-3}, \eqref{7.1}, and \eqref{251110-5} to obtain
\begin{align*}
\|v\|_{\calK_h}^2 + h\|v\|_{\p \calK_h}^2
\lesssim h^{-1}\|v\|_{\calT_h}^2
\lesssim \|\nab_{\Gamma_h}v\|_{\calK_h}^2 + h\|\nab v\|_{\calT_h}^2\lesssim \|v\|_h^2 +  \|\nab_{\Gamma_h}v\|_{\calK_h}^2,
\end{align*}
where we used the continuity of $v$ on $\Omega_h$.

{\bf Step 3:}
We extend $\bn_h$ by constant on each element
so that for each $T \in \calT_h$, ${\bf P}_{h}$ is constant and $\nab_{\Gamma_h}v$ is well-defined and is a piecewise polynomial with respect to $\calT_h$. By the inverse equalities \eqref{inverse-est-4}, \eqref{inverse-est-2}, the large intersection covering property, and \eqref{F-1}, we have
\begin{align}
\begin{split} \label{251110-3}
h\|\nab_{\Gamma_h}v\|_{\p\calK_h}^2
& \lesssim h^{-1}\|\nab_{\Gamma_h}v\|_{\calT_h}^2
\lesssim h^{-1}\sum_{x \in \calX_h}\|\nab_{\Gamma_h}v\|_{\calT_{h,x}}^2 \\
& \lesssim h^{-1}\sum_{x \in \calX_h}\|\nab_{\Gamma_h}v\|_{T_x}^2 + \|[\nab_{\Gamma_h}v]\|_{\calF_h}^2 + h^2\|[\Grad(\nab_{\Gamma_h}v)\bnu_F]\|_{\calF_h}^2 \\
& =: I + II + III.
\end{split}
\end{align}
We will estimate each term separately.

\textit{Estimate I:} Let $ \bg \in \mathbb{P}_0^{dc}(\calT_h)$ satisfy $\bg|_T = \frac{1}{|T|}\int_T\nab_{\Gamma_h}v$ for all $T \in \calT_h$. A standard Poincare inequality yields
\[
\|\nab_{\Gamma_h}v - \bg\|_T
\lesssim h\|\Grad(\nab_{\Gamma_h}v)\|_T,
\]
and moreover, the inverse inequality \eqref{inverse-est-3} implies
\[
\|\nab_{\Gamma_h}v - \bg\|_K^2
\lesssim h\|\Grad(\nab_{\Gamma_h}v)\|_T^2.
\]
Using these results with the large intersection covering property, we find
\begin{align*}
I
& \lesssim h^{-1}\sum_{x \in \calX_h}(\|\nab_{\Gamma_h}v - \bg\|_{T_x}^2 + \|\bg\|_{T_x}^2)
\lesssim h^{-1}\sum_{x \in \calX_h}(\|\nab_{\Gamma_h}v - \bg\|_{T_x}^2 + h\|\bg\|_{K_x}^2) \\
& \lesssim h^{-1}\sum_{x \in \calX_h}(\|\nab_{\Gamma_h}v - \bg\|_{T_x}^2 + h\|\bg - \nab_{\Gamma_h}v\|_{K_x}^2 + h\|\nab_{\Gamma_h}v\|_{K_x}^2) \\
& \lesssim h\|\Grad(\nab_{\Gamma_h}v)\|_{\calT_h}^2 + \|\nab_{\Gamma_h}v\|_{\calK_h}^2.
\end{align*}
Based on \cite[(2.5)]{neilan2025c0} and notice that $\bn_h$ is piecewise constant, we have
\begin{align}
\Grad(\nab_{\Gamma_h}v)
& = {\bf P}_h\nab^2v. \label{251110-2}
\end{align}
Therefore, the boundedness of ${\bf P}_h$ and \eqref{lem-7.2} lead to
\begin{align}
I
& \lesssim h^2\|\Delta_{\Gamma_h}v\|_{\calK_h}^2 + \|[\nab^2v]\|_{\calF_h}^2 + \|\nab_{\Gamma_h}v\|_{\calK_h}^2\lesssim \|v\|_h^2 + \|\nab_{\Gamma_h}^2 v\|_{\calK_h}^2. \label{251110-4}
\end{align}
\textit{Estimate II:} It follows from the similar arguments as in the proof of \eqref{251110-1} as well as \eqref{251110-5} that
\begin{equation}\label{IIBound}
\begin{split}
II
&\lesssim \|[{\bf P}_h]\{\nab v\}\|_{\calF_h}^2 + \|\{{\bf P}_h\}[\nab v]\|_{\calF_h}^2\\
&\lesssim \|[{\bf P}_h - {\bf P}]\{\nab v\}\|_{\calF_h}^2 + \|[\nab v]\|_{\calF_h}^2
\lesssim h\|\nab v\|_{\calT_h}^2 + \|[\nab v]\|_{\calF_h}^2
\lesssim \|v\|_{h}^2 + h^2 \|\nab_{\Gamma_h} v\|_{\calK_h}^2.
\end{split}
\end{equation}

\textit{Estimate III:} Using \eqref{251110-2}, \eqref{251110-1} and its proof again, also with \eqref{lem-7.2}, we have
\begin{align}
\begin{split} \label{251110-6}
III
& = h^2\|[{\bf P}_h\nab^2v\bnu_F]\|_{\calF_h}^2 \\
& \lesssim h^2(\|[{\bf P}_h\nab^2v]\{\bnu_F\}\|_{\calF_h}^2 + \|\{{\bf P}_h\nab^2v\}[\bnu_F]\|_{\calF_h}^2) \\
& \lesssim h^2\|[{\bf P}_h\nab^2v]\|_{\calF_h}^2 + h^2\|\nab^2v\|_{\p\calT_h}^2 \\
& \lesssim h^3\|\nab^2v\|_{\calT_h}^2 + h^2\|[\nab^2v]\|_{\calF_h}^2 + h\|\nab^2v\|_{\calT_h}^2 \\
& \lesssim h^2\|\Delta_{\Gamma_h}v\|_{\calK_h}^2 + \|[\nab^2v]\|_{\calF_h}^2\lesssim \|v\|_h^2.
\end{split}
\end{align}
Combining \eqref{251110-3}, \eqref{251110-4}, \eqref{IIBound}, and \eqref{251110-6} yields
\begin{equation}
\label{Step3}
h \|\nab_{\Gamma_h} v\|_{\p \calK_h}^2\lesssim \|v\|_h^2 + \|\nab_{\Gamma_h} v\|_{\calK_h}^2.
\end{equation}

{\bf Step 4:}
Combining Steps 1--3, we have shown
\begin{align}\label{eqn:PrelimAllAtOnce}
h\|\nab v\|_{\calT_h}^2+ \|v\|_{\calK_h}^2 +h \|v\|_{\p \calK_h}^2 
+ h \|\nab_{\Gamma_h} v\|_{\p \calK_h}^2 \lesssim \|v\|_h^2 + \|\nab_{\Gamma_h} v\|_{\calK_h}^2.
\end{align}

We integrate by parts, apply the Cauchy-Schwarz inequality, and
the estimate \eqref{eqn:PrelimAllAtOnce} to obtain
\begin{align*}
\|\nab_{\Gamma_h} v\|_{\calK_h}^2  
&= -(\Delta_{\Gamma_h} v,v)_{\calK_h}+([\nab_{\Gamma_h} v]\cdot \bmu_h,v)_{\calE_h}\\
&\lesssim  \|v\|_h \big( \|v\|_{\calK_h} + h^{1/2}\|v\|_{\p \calK_h}\big)\\
&\lesssim  \|v\|_h^2+ \|v\|_h \|\nab_{\Gamma_h} v\|_{\calK_h}.
\end{align*}
Thus, we have 
\begin{equation}\label{Step4}
\|\nab_{\Gamma_h} v\|_{\calK_h}^2  \lesssim \|v\|_h^2.
\end{equation}

{\bf Step 5:}
We use \eqref{inverse-est-2}--\eqref{inverse-est-4} and \eqref{lem-7.2}:
\begin{equation}\label{Step5}
h^2\|\nab_{\Gamma_h}^2v\|_{\calK_h}^2 + h^3\|\nab_{\Gamma_h}^2v\|_{\p\calK_h}^2
\lesssim h\|\nab^2v\|_{\calT_h}^2
\lesssim h^2\|\Delta_{\Gamma_h}v\|_{\calK_h}^2 + \|[\nab^2v]\|_{\calF_h}^2
\lesssim \|v\|_h^2.
\end{equation}
Finally, we combine \eqref{eqn:PrelimAllAtOnce}--\eqref{Step5} to complete the proof.

\end{document}